%% file: thomp-genus0.tex
\documentclass{pspum-l}
\usepackage{amsmath, amsthm, amscd, amsfonts,amssymb} 

\numberwithin{equation}{section}

\input macroltx.tex

\newcommand{\C}{{\text{\rm C}}}

\newcommand{\sh}{{\text{\bf sh}}}
\renewcommand{\phi}{{\varphi}}
 \newcommand{\Cu}{\text{Cu}}
\newcommand{\Prob}[1]{\text{Problem$_{#1}^{\!\scriptscriptstyle{g=0\!}}$}}
\newcommand{\D}{{\text{DP}}}

\begin{document} 
\font\eightrm=cmr8  \font\eightit=cmsl8 \let\it=\sl  

\title[genus 0 problems]{Relating two genus 0 problems of John
Thompson}

\author[M.~Fried]{Michael D.~Fried}  
\date{\today} 
\newcommand\rk{{\text{\rm rk}}} 
\newcommand{\TG}[2]{{{}_{#1}^{#2}\tilde{G}}} 
\newcommand{\vsmatrix}[4]{{\left(\smallmatrix #1 & #2 \\ #3 & #4  
\\ \endsmallmatrix \right)}} 
\newcommand{\tbg}[1]{{{}^{#1}\tilde\bg}} 
\newcommand{\Inn}{{\text{\rm Inn}}} 
\newcommand{\af}{{\text{\rm af}}} 
\newcommand{\f}{{\text{\rm f}}} 
\newcommand{\Spin}{{\text{Spin}}}
 
\address{Emeritus, University of California at Irvine} 
\email{mfried\@math.uci.edu}  

\begin{abstract} Excluding a precise list of groups like alternating, symmetric, cyclic and
dihedral, from 1st year algebra (\S\ref{nonSporadics}),  we expect there are only
finitely many monodromy groups of primitive genus 0 covers. Denote this nearly  proven genus 0
problem     as \Prob 2. We call the exceptional groups  {\sl 0-sporadic\/}. Example: Finitely many
Chevalley groups are 0-sporadic. A proven result: Among {\sl polynomial\/} 0-sporadic groups,
precisely three  produce covers falling in nontrivial {\sl reduced families\/}. Each 
(miraculously) defines one natural genus 0
$\bQ$ cover of the 
$j$-line. The latest Nielsen class techniques apply to these
dessins d'enfant to see their subtle arithmetic and interesting
cusps. 

John Thompson earlier considered another genus 0 problem: 
To find $\theta$-functions uniformizing certain genus 0 (near) modular curves. We call 
this \Prob 1. We pose uniformization problems for $j$-line covers
in two cases. First: From the three 0-sporadic examples of \Prob2. Second: From finite 
collections of genus 0 curves with aspects of \Prob1.    

\end{abstract}


\thanks{Thanks to NSF 
Grant
\#DMS-0202259, support for the {\sl Thompson\/} Semester at U.~of Florida} 

\maketitle 
\def\addcontentsline#1#2#3{%
\addtocontents{#1}{\protect\contentsline{#2}{#3}{}}}

\setcounter{tocdepth}{2}
\!\!\!\!\tableofcontents

\section{Genus 0 themes} We denote projective 1-space 
$\prP^1$ with a specific uniformizing variable
$z$ by
$\prP^1_z$. This decoration helps track distinct domain and
range copies of $\prP^1$. We use classical groups: $D_n$ (dihedral), $A_n$ (alternating) and $S_n$
(symmetric) groups of degree $n$;
$\PGL_2(K)$, M\"obius transformations over
$K$; and  generalization of these to $\PGL_{u+1}(K)$ acting on $k$-planes, $0\le k\le u-1$, of 
$\prP^u(K)$ ($K$ points of projective $u$-space). \S\ref{bfineMod} denotes the space of four
distinct unordered points of
$\prP^1_z$ by $U_4=((\prP_z^1)^4\setminus \Delta_4)/S_4$. For $K$ a field, $G_K$ is its
absolute Galois group  (we infrequently allude to this for some applications). 

A  $g\in S_n$ has an index
$\ind(g)=n-u$ where $u$ is the number of disjoint cycles in $g$. Example: $(1\,2\,3)(4\,5\,6\,7)\in
S_7$ has index $7-2=5$. 
Suppose $\phi: X\to \prP^1_z$ is a degree
$n$ cover
(of compact Riemann surfaces). We assume the reader knows about the genus $g_X$ of 
$X$ given a {\sl branch cycle description\/} $\bg=(\row g r)$ for $\phi$ (\S\ref{standEquiv}):
$2(n+g_X-1)=\sum_{i=1}^r
\ind(g_i)$ (\cite[\S2.2]{VB} or \cite[Chap.~4]{FB}). 

\subsection{Production of significant genus 0 curves} 
Compact Riemann surfaces arose to codify two variable algebraic relations. The moduli of covers is
a refinement. For a given genus, this refinement has many 
subfamilies, with associated discrete invariants. Typically, these invariants are some type of
{\sl Nielsen class\/} (\S\ref{nc1}).  The parameter space for such a family can have any dimension.
Yet, we benefit by comparing  genus $g$ moduli  with cases where the cover moduli  has dimension 1.
The gains come by detecting the moduli resemblances to, and differences from,   modular
curves. 

We emphasize: Our technique produces a parameter for families of equations
from an essential defining property of the equations. We aim 
for a direct description of that parameter using the defining property. These
examples connect two themes useful for intricate work on families of equations. We
refer to these as two {\sl genus 0 problems\/} considered  by John Thompson. Our first version is a
naive form.

\begin{edesc}  \label{prob0}
\item \label{prob0a} \Prob1: If a moduli space of algebraic relations is a genus 0
curve, where can we find  a uniformizer for it? 
\item \label{prob0b} \Prob2: Excluding symmetric, alternating, cyclic and dihedral groups,  what
others are monodromy groups for primitive genus 0 covers?
\end{edesc} 
Our later versions of each statement explicitly connect with well-known
problems.  
\cite{RETExp},  \cite{Fr-Schconf}, \cite{GMS}  show examples  
benefiting from the {\sl monodromy method}. 

We use the latest Nielsen class techniques (\S\ref{nc1}; excluding the {\sl
shift-incidence matrix\/}   from \cite[\S9]{BFr} and \cite{schurtype}) to understand these
parameter spaces as natural 
$j$-line covers. They are not modular curves, though emulating \cite{BFr} we observe modular
curve-like properties. 

Applying the Riemann-Roch Theorem to \Prob1 does not actually answer the
underlying question. Even when (say, from  Riemann-Hurwitz) we find  a curve has genus 0  
that doesn't trivialize  uniformizing its function field. Especially when the moduli space has
genus 0: We seek  a uniformizer defined by the moduli problem. That the $j$-line covers of our examples have genus 0 allows them to effectively
parametrize (over a known field) solutions to problems with a considerable literature.
We justify that \wsp albeit, briefly \wsp to give weight to our choices. 

\cite{FKraTheta}  uses
$k$-division 
$\theta$-nulls (from elliptic curves) to uniformize certain modular curves. Those
functions, however, have
nothing to do with the moduli for our examples. Higher dimensional 
$\theta$-nulls on the (1-dimensional) upper-half plane are akin to, but not the same as, what 
quadratic form people  call
$\theta$-functions. The former do appear in our examples of \Prob1 (\S\ref{MMoon}; 
we explain more there on this $\theta$-confusing point). We are new at Monstrous Moonshine, though
the required expertise documented by \cite{Ray} shows we're not alone. Who can predict from where
significant uniformizers will arise? If a
$\theta$-null intrinsically attaches to the moduli problem, we'll use it. 

\subsection{Detailed results} \S\ref{nonSporadics} has the precise definition of 0-sporadic
(also, polynomial 0-sporadic and the general $g$-sporadic).  All modular curves appear  as (reduced;
see
\S\ref{nc2}) families of genus 0 covers \cite[\S2]{FrGGCM}. Only, however,  finitely many modular
curves have genus 0. Our first examples are moduli spaces for polynomial 0-sporadic  groups 
responding to
\eql{prob0}{prob0b}. These moduli spaces are genus 0 covers of the $j$-line, responding to 
\eql{prob0}{prob0a}, yet they are not modular curves. 

\subsubsection{The moduli of three 0-sporadic monodromy groups}   Three polynomial 0-sporadic
groups stand out on M\"uller's list (\S\ref{MullList}): These have 
degrees
$n=7, 13$ and
$15$, with {\sl four\/} branch points (up to reduced equivalence \S\ref{nc2}). Their families have
genus 0 suiting question 
\eql{prob0}{prob0a}. Each case sums up in {\sl one\/} (for each
$n\in \{7,13,15\}$) genus 0 
$j$-line cover ($\psi_n: X_n\to \prP^1_j$) over $\bQ$. We tell much about these spaces,
their {\sl b-fine moduli\/} properties and their cusps (Prop.~\ref{H7comp} and
Prop.~\ref{H13comp}). 

We stress the uniqueness of $\psi_n$, and its $\bQ$ structure.    
Reason: The moduli problem defining it does {\sl not\/} produce polynomials over
$\bQ$. Let $K_{13}$ be the unique degree 4 extension of $\bQ$ in $\bQ(e^{2\pi i/13})$. For $n=13$,  a parameter uniformizing 
$X_{13}$ as a
$\bQ$ space gives coordinates for the four (reduced) families of polynomials over $K_{13}$. These 
appear as solutions of Davenport's problem (\S\ref{DPs}). Resolving 
Davenport's problem (combining  group theory and arithmetic in
\cite{FrRedPol}, \cite{feit}, \cite{RETExp} and  
\cite{Fr-Schconf}) suggested that genus 0 covers have a limited set of monodromy
groups (\Prob2).  
\cite[\S5]{Fr-Schconf} and \cite[App.~D]{FrExcTow} has more on the  applications.

\subsubsection{Modular curve-like genus 0 and 1 curves} \label{genus01Mod} Our second example is
 closer to classical modular curve themes, wherein uniformizers of certain genus 0 curves
appear  from $\theta$-functions. \S\ref{MMoon} briefly discusses Monstrous Moonshine for
comparison. Our situation is the easiest rank 2 {\sl Modular Tower}, defined by the group
$F_2\xs
\bZ/3$, with $F_2$ a free group on two generators. We call this the $n=3$ case. For
each prime
$p\not =3$, and for each integer
$k\ge 0$, there is a  map 
$\bar \psi_{p,k}: \bar \sH_{p,k}^\rd\to \prP^1_j$ with $\bar \sH_{p,k}^\rd$ a (reduced) moduli
space. The gist of Prop.~\ref{H3NC}: Each such  $\bar \sH_{p,k}^\rd$ {\sl is\/} nonempty (and $\bar
\psi_{p,k}$ is a natural $j$-line cover). For a given
$p$ the collection 
$\{\bar \sH_{k,p}^\rd\}_{k\ge 0}$ forms a projective system; we use this below. 

We contrast the $n=3$ case with the 
case $F_2\xs
\bZ/2$. This is the $n=2$ case: $-1\in \{\pm 1\}=\bZ/2$ maps generators of $F_2$ to their
inverses.  We do this to give the {\sl Modular Tower\/} view of noncomplex multiplication in Serre's
{\sl Open Image Theorem\/} (\cite[IV-20]{SeAbell-adic}).  
The gist of Prop.~\ref{H2NC}: 
Serre's Theorem covers less territory than might be expected. 
\cite[\S5.2-5.3]{FrExcTow} applies this to producing genus 0 {\sl exceptional\/}
covers. This shows Davenport's problem is not an isolated example.    

The (strong) Main
Conjecture on Modular Towers \cite[\S1.2]{schurtype} says the following for $n=3$.  Only finitely
many 
$\bar \sH_{p,k}^\rd\,$s have a genus 0 or 1 (curve) component. These are moduli spaces, and rational
points on such components interpret significantly for many problems. For $n=2$ the
corresponding spaces are modular curves, and they have but one component. Known values of
$(p,k)$ where $\bar \sH_{p,k}^\rd$ has more than one component include $p=2$, with $k=0$ and 1.
\S\ref{projNCs} explains the  genus 0 and 1 components for the second
of these. For  $j$-line covers coming from Nielsen classes there is a
map from elements in the Nielsen class to cusps. 
The most modular curve-like property of these spaces is that they fall in sequences attached 
to a prime $p$. Then, especially significant are the {\sl g-$p'$ cusps\/}  
(\S\ref{cusptype}). 

Most studied  of the g-$p'$ cusps are those we call {\sl Harbater-Mumford}. For example, in this language the
width $p$ (resp.~1) cusp on the modular curve 
$X_0(p)$ ($p$ odd)  is the  Harbater-Mumford
(resp.~shift of a Harbater-Mumford) cusp (Ex.~\ref{HMreps}). Prob.~\ref{projNielCusp} is a  
conjectural refinement of Prop.~\ref{H2NC}. This distinguishes those components containing H-M cusps
among the collection of all components of $\bar \sH_{p,k}^\rd\,$s. 

If right, we can expect
applications for those components  that parallel \cite{SeAbell-adic} (for  $n=2$).   
We conclude with connections between  \Prob1 and \Prob2. This gives an historical context for  
using   cusps of $j$-line covers from Nielsen classes.
 \begin{edesc} \label{prob1-prob2}
\item \label{prob1-prob2a} Comparison of our computations with computer construction of
equations for the Davenport pair families in \cite{couveignes2} (\S\ref{expEqs}).
\item \label{prob1-prob2c} The influence of John Thompson on \Prob1 and \Prob2 (\S\ref{Thomp}). 
\end{edesc}  
    
\section{Examples from \Prob 2} \label{MullerDPs} We briefly state Davenport's problem and review
Nielsen classes. Then, we explain Davenport's problem's special place among 
polynomial 0-sporadic groups. 

\subsection{Review of Davenport's problem} \label{DPs} The name
{\sl Davenport pair\/} (now called S(trong)DP) first referred to pairs $(f,g)$ of polynomials, over
a number field
$K$ (with ring of integers
$\sO_K$) satisfying this. 
\begin{triv} \label{rangeSt} Range equality:
$f(\sO/\bp)=g(\sO/\bp)$ for almost all prime ideals
$\bp$ of $\sO_K$.
\end{triv} \noindent  Davenport asked this question just for polynomials over $\bQ$. We also assume  there should be no linear change of variables (even over $\bar
K$) equating the polynomials. This is an hypothesis that we intend from this point.  There is a complete description of the Davenport pairs where $f$ is indecomposable (\S\ref{nc1}). 

In 
this case such polynomials are {\sl i(sovalent)DP\/}s: Each value in the range of $f$ or $g$ is
achieved with the same multiplicity by both polynomials. As in
\cite[7.30]{AFH}, this  completely  describes all such pairs even with a weaker
hypothesis: 
\eqref{rangeSt} holds for just $\infty$-ly many prime ideals of $\sO_K$. 

\subsection{Review of Nielsen classes} \label{nc1} A Nielsen class is
a combinatorial invariant attached to a (ramified) cover $\phi: X\to \prP^1_z$ of compact
Riemann surfaces. If $\deg(\phi)=n$, let $G_\phi\le S_n$ be the monodromy group of $\phi$.
The cover is {\sl primitive\/} or {\sl indecomposable\/} if the following equivalent properties
hold.
\begin{edesc} \item It has no decomposition  $X\mapright{\phi'} W \mapright{\phi''}
\prP^1_z$, with 
$\deg(\phi')\ge 2$, 
$\deg(\phi'')\ge 2$.
\item $G_\phi$ is a primitive subgroup of $S_n$. 
\end{edesc} 

Let 
$\bz$ be the branch points of $\phi$,  
$U_\bz=\prP^1_z\setminus
\{\bz\}$ and $z_0\in U_\bz$. Continue points over $z_0$ along
paths based at $z_0$, having the following form: $\gamma\cdot \delta_i\gamma^{-1}$, $\gamma, \delta$
on
$U_\bz$ and $\delta_i$ a small clockwise circle around $z_i$. This attaches to $\phi$ a collection
of conjugacy classes $\bfC=(\row {\text{C}} r\}$, one for each $z_i\in \bz$. The
associated  {\sl  Nielsen  
class\/}: $$\ni=\ni(G,\bfC)=\{\bg= (\row g r) \mid g_1\cdots g_r=1, 
\lrang{\bg}=G \text{\ 
and\ } \bg\in \bfC\}.$$ {\sl Product-one\/} is the name for the condition $g_1\cdots g_r=1$. 
From it come invariants attached to spaces defined by Nielsen classes.
{\sl Generation\/} is the name of condition $\lrang{\bg}=G$. Writing
$\bg\in\bfC$ means  the 
$g_i\,$s define conjugacy classes in
 $G$, possibly in another order, the  same (with multiplicity) as 
those in 
$\bfC$. So, each cover $\phi: X\to\prP^1_z$ has a uniquely attached Nielsen class: 
$\phi$ is in the Nielsen class $\ni(G,\bfC)$. 
 
\subsubsection{Standard equivalences} \label{standEquiv}
Suppose we have  $r$ (branch) points $\bz$, and a corresponding choice $\bar \bg$ of {\sl classical 
generators\/} for 
$\pi_1(U_\bz,z_0)$ \cite[\S1.2]{BFr}. Then, $\ni(G,\bfC)$ lists all 
homomorphisms 
from $\pi_1(U_\bz,z_0)$ to $G$. These give a cover with branch points $\bz$ 
associated to $(G,\bfC)$. Elements of $\ni(G,\bfC)$ are {\sl branch cycle
descriptions\/} for these covers relative to $\bar\bg$.  
Equivalence classes of 
covers with a fixed set of branch points $\bz$, correspond one-one to equivalence classes on 
$\ni(G,\bfC)$. We caution: Attaching a Nielsen class
representative to a cover requires picking one from many possible
$r$-tuples $\bar \bg$. So, it is not an algebraic process.

\cite[\S3.1]{BFr} reviews common equivalences with examples and relevant definitions. 
such as the group
$\sQ''$ below. Let
$N_{S_n}(G,\bfC)$ be those $g\in S_n$ normalizing $G$ and permuting the
collection of conjugacy classes in $\bfC$. Absolute (resp.~inner) equivalence classes of
covers (with branch points at
$\bz$) correspond to the elements of
$\ni(G,\bfC)/N_{S_n}(G,\bfC)=\ni(G,\bfC)^\abs$ (resp.~$\ni(G,\bfC)/G=\ni(G,\bfC)^\inn$). Especially
in 
\S\ref{7-13-15} we use {\sl absolute}, {\sl inner\/} and for each of these {\sl reduced\/}
equivalence. These show how to compute specific properties of 
$\sH(G,\bfC)^\abs$, $\sH(G,\bfC)^\inn$ and their reduced versions, parametrizing the
equivalences classes of covers as $\bz$ varies. 

\subsubsection{Reduced Nielsen classes} \label{nc2} Reduced
equivalence  corresponds each cover $\phi: X\to \prP^1_z$ to
$\alpha\circ \phi: X\to \prP^1_z$, running over $\alpha\in \PGL_2(\bC)$. 
 If $r=4$, a nontrivial equivalence arises because for any $\bz$ there is a Klein 4-group in
$\PGL_2(\bC)$ mapping $\bz$ into itself. (An
even larger group leaves special, {\sl elliptic\/}, $\bz$ fixed.) This interprets as an equivalence
from   a Klein 4-group $\sQ''$ acting on Nielsen classes (\S\ref{deg7j-line}). 
 Denote associated absolute 
(resp.~inner) reduced
Nielsen class representatives  by   
$$\ni(G,\bfC)/\lrang{N_{S_n}(G,\bfC),\sQ''}\!=\!\ni(G,\bfC)^{\abs,\rd}
(\text{resp.~}\!\!\ni(G,\bfC)/\lrang{G,\sQ''}\!=\!\ni(G,\bfC)^{\inn,\rd}).$$  

These give formulas for branch cycles presenting
$\sH(G,\bfC)^{\abs,\rd}$ and $\sH(G,\bfC)^{\inn,\rd}$ as upper half plane quotients 
by a finite index subgroup of $\PSL_2(\bZ)$. This is a ramified cover of the classical $j$-line 
branching over the traditional places (normalized in \cite[Prop.~4.4]{BFr} to 
$j=0,1,\infty$).  Points over $\infty$ are {\sl meaningfully\/}  called cusps. Here is an
example of how we will use these. 

\S\ref{deg7j-line} computes from these tools  two $j$-line covers (dessins d'enfant)  conjugate
over $\bQ(\sqrt{-7})$ parametrizing reduced classes of degree 7 Davenport polynomial pairs. In
fact, the (by hand) Nielsen class computations  show the covers are equivalent over
$\bQ$. This same phenomenon happens for all pertinent degrees $n=7,13,15$, though the field
$\bQ(\sqrt{-7})$ changes and corresponding Nielsen classes have subtle differences. You can see
these by comparing $n=7$ with $n=13$ (\S\ref{deg13j-line}). 

\subsection{Davenport Pair monodromy groups}  \label{diffSets} Let $u\ge 2$. A {\sl Singer cycle\/}
is a  generator 
$\alpha$ of $\bF_{q^{u+1}}^*$, acting by multiplication as a 
matrix through identifying  
$\bF_q^{u+1}$ and $\bF_{q^{u+1}}$.  Its image in 
$\PGL_{u+1}$  acts on points and
hyperplanes of $\prP^u(\bF_q)$. 

Let $G$ be a group 
with two doubly transitive representations $T_1$ and $T_2$,  equivalent as group representations,
yet not permutation equivalent, and with 
$g_\infty\in G$  an
$n$-cycle in $T_i$, $i=1,2$. Excluding  the well-documented degree 11 case, $G$ has 
these properties (\cite{feit}, \cite{RETExp},
\cite[\S8]{Fr-Schconf}) with   
$\C_\alpha$ the conjugacy class of $g_\infty$.   
\begin{edesc} \label{Davgroups} \item $G\ge \PSL_{u+1}(\bF_q)$; $T_1$ and $T_2$ 
act on points and hyperplanes of
$\prP^u$. 
\item  $n=(q^{u\np1}-1)(q-1)$ and $g_\infty$ is a Singer
$n$-cycle. 
\end{edesc}  

\subsubsection{Difference sets} 
Here is how difference sets (\S\ref{diffSet})
appear from \eqref{Davgroups}. 

\begin{defn} Call $\sD\le \bZ/n$ a {\sl difference set\/} if nonzero differences from 
$\sD$ distribute evenly over $\bZ\setminus \{0\}$. The multiplicity $v$ of the appearance of 
each element is the multiplicity of $\sD$. Regard a  difference set and any translate of it as
equivalent.
\end{defn} Given the linear representation  from $T_1$ on
$\row x n$, the representation $T_2$ is on  $\{\sum_{i\in \sD+j} x_i\}_{j=1}^n$ with 
$\sD$ a difference set. The multiplier group $M_n$ of $\sD$ is  
$$\{m\in (\bZ/n)^*\mid m\cdot\sD=\sD+j_m, \text{ with }j_m\in \bZ/n\}.$$ We say $m\cdot \sD$
is equivalent to $\sD$ if $m\in M_n$. In $\PGL_{u+1}(\bF_q)$,  $\alpha^m$ is 
conjugate to $\alpha$ exactly when  $m\in M_n$. The difference set
$-\sD$ corresponds to an interchange between the representations on points and hyperplanes. 
Conjugacy classes in
$\PGL_{u+1}(\bF_q)$ of powers of
$\alpha$ correspond one-one to difference sets  equivalence classes mod $n$. 

\subsubsection{Davenport pair Nielsen classes} We label our families of polynomials by an 
$m\in (\bZ/n)^*\setminus M_n$ that multiplies the difference set to an inequivalent difference set.
Our families are  of absolute  reduced classes of covers in a Nielsen class. 
Conjugacy classes have the form
$(\row {\C} {r-1},\C_\alpha)$, the groups satisfy $G\ge \PSL_{u+1}(\bF_q)$, and covers in the
class have genus 0. Two results of Feit show $r-1\le 3$: 
\begin{edesc} \item $\ind(\C)\ge n/2$ if $\C$ is a conjugacy classes of 
$\PGL_{u+1}(\bF_q)$ \cite{feit1}. 
\item $\ind(\C)\ge n(q-1)/q$ if $\C$ is a conjugacy class in
$\text{P}\Gamma\text{L}_{u+1}(\bF_q)$ \cite{feit2}. 
\end{edesc}. 
Conclude: If $n$ is odd, then $r-1=3$ implies the following for respective cases. 
\begin{edesc} \label{involutions} \item \label{involutions7} $n=7$: $\C_i\,$s are in the
conjugacy class of transvections  (fixing a hyperplane), with index 2. So, they are all
conjugate. 
\item \label{involutions13} $n=13$: $\C_i\,s$ are in the conjugacy class of elements fixing a 
a hyperplane (determinant -1),  so they generate $\PGL_3(\bZ/3)$ and all are conjugate. 
\item \label{involutions15} $n=15$: Two of the $\C_i\,$s are in the conjugacy class of
transvections, and one fixes just a line. 
\item $n\ne 7,13,15$: $r=3$.
\end{edesc}
Transvections in $\GL_{u+1}$ have the form $\bv\mapsto
\bv+\mu_H(\bv)\bv_0$ with
$\mu_H$ a linear functional with kernel a hyperplane $H$, and $\bv_0\in H\setminus
\{0\}$
\cite[p.~160]{eartin}.
 For $q$ a power of two, these are involutions: 
exactly those fixing points of a hyperplane. For $q$ odd, 
involutions fixing the points of a hyperplane (example, induced by a reflection in $\GL_{n+1}$ in
the hyperplane) have the maximal number of fixed points. When
$q$ is a power of 2, there are involutions fixing precisely one line. Jordan normal form 
shows these are conjugate to $$\begin{pmatrix} 1&0&0&0\\ a&1&0&0 \\ 0& b & 1 & 0\\ 0 & 0
&c&1\end{pmatrix}.$$ This is an involution if and only if $ab=bc=0$. So, either $b=0$ or 
$a=c=0$. Only in the latter case is the fixed space a line. So, given the conjugacy class of
$g_\infty$, only one possible Nielsen class defines polynomial Davenport pairs when $n=15$. 

\subsection{M\"uller's list of primitive polynomial monodromy groups} \label{MullList} We 
reprise M\"uller's list of the polynomial 0-sporadic groups   
(\cite{primPol}). Since such a group comes from a primitive cover, it goes with a primitive
permutation representation. As in \S\ref{nonSporadics} we regard two inequivalent representations
of the same group as different 0-sporadic groups. We emphasize how pertinent was
Davenport's  problem.    Exclude  (finitely many) groups with simple core 
$\PSL_2(\bF_q)$ (for very small $q$) and the Matthieu groups of degree 11 and 23. Then, all 
remaining  groups from his list are from 
\cite{FrRedPol} and have properties  \eqref{Davgroups}. 
\cite[\S9]{Fr-Schconf} reviews and completes this. These six polynomial 0-sporadic groups (with
corresponding Nielsen classes) all give  Davenport pairs. 
We concentrate on those three having one extra
property:   
\begin{triv} \label{dim1mod} Modulo
$\PGL_2(\bC)$ (reduced equivalence as in \S\ref{nc2})  action, the space of these
polynomials has dimension at least (in all cases, equal) 1.  \end{triv}

We restate the properties shown above for these polynomial covers.
\begin{itemize} \item They have degrees  
from
$\{7,13,15\}$ and $r=4$.
\item All  $r\ge 4$ branch point indecomposable polynomial maps in an iDP pair
(\S\ref{MullerDPs}) are in one of the respectively, 2, 4 or 2 Nielsen classes corresponding to the
respective degrees 7, 13, 15.
\end{itemize} 
\cite{FrRedPol} outlines this.  \cite[\S2.B]{RETExp} uses it to explain Hurwitz
monodromy action. 

Let 
$\sH_7^\D$, $\sH^\D_{13}$ and $\sH^\D_{15}$ denote the spaces of polynomial covers that are
one from a Davenport pair having four branch points (counting
$\infty$). The subscript decoration corresponds to the respective
degrees. We assume absolute, reduced equivalence (as in \S\ref{nc2}). 

\section{Explanation of the components for $\sH_7^\D,\sH^\D_{13},\sH^\D_{15}$}
\label{7-13-15} The analytic families of respective degree $n$ polynomials fall into
several components. Each component, however, corresponds to a
different Nielsen class. For example,  $\sH_7^\D$, the space of degree 7 Davenport polynomials
has two components: with a polynomial associated to a polynomial in the other as
a Davenport Pair. 
\subsection{Explicit difference sets}  Often we apply Nielsen classes to problems about the
realization of covers over
$\bQ$. Then, one must assume
$\bfC$ is rational. \cite[Thm.~2]{FrRedPol} proved (free
of the finite simple group classification) that {\sl no\/}  indecomposable polynomial DPs could
occur over
$\bQ$.  

There are polynomial covers in our Nielsen classes.  
So, $\bg\in \ni(G,\bfC)$ has an $n$-cycle entry, 
$g_\infty$. These conjugacy classes for all $n$ are similar: 
$$\bfC_{n,u;k_1,k_2,k_3}=(\C_{2^{k_1}},\C_{2^{k_2}},\C_{2^{k_3}},\C_{(\alpha)^u}),$$
where
$\C_{2^k}$ denotes a (nontrivial) conjugacy class of involutions of index $k$ 
and
$(u,n)=1$.  We explain the case $n=7$ in the following rubric. 

\begin{edesc} \label{propDiff} \item \label{propDiffa} Why $\C_{2^{k_i}}=\C_{2^2}$,
$j=1,2,3$,  is the conjugacy class of a transvection (denote the resulting conjugacy
classes by $\bfC_{7,u;3\cdot 2^2}$).   
\item \label{propDiffb} Why the two components of $\sH_7^\D$  are
$$\sH_{+}=\sH(\PSL_3(\bZ/2),\bfC_{7,1;3\cdot 2^2})^{\abs,\rd} \text{ and }
\sH_{-}=\sH(\PSL_3(\bZ/2),\bfC_{7,-1;3\cdot 2^2})^{\abs,\rd}.$$ 
\item \label{propDiffc} Why the closures of $\sH_{\pm}$ over $\prP^1_j$ (as natural
$j$-line covers) are equivalent genus 0, degree 7 covers  over $\bQ$.  
\item \label{propDiffd} Why $\sH_{\pm}$, as degree 7 Davenport moduli,  have
definition field $\bQ(\sqrt{-7})$. 
\end{edesc}   

\subsection{Difference sets give properties \eql{propDiff}{propDiffa} and 
\eql{propDiff}{propDiffb}} \label{diffSet} 
\cite[Lem.~4]{FrRedPol}  normalizes Nielsen class representatives
$(g_1,g_2,g_3,g_\infty)$ for DP covers so that in both representations $T_{j,n}$,
$j=1,2$, $g_\infty=(1\,2\,\dots\, n)^{-1}$ identifies with some allowable
$\alpha_n$. 

Regard
$(g_\infty)T_{1,n}$ ($g_\infty$ in the representation $T_{1,n}$) 
as translation by -1 on $\bZ/n$. Then, $(g_\infty)T_{2,n}$ is translation by -1 on the collection
of sets 
$\{\sD+c\}_{c\in \bZ/n}$. Take $v$ to be the multiplicity of $\sD$. Then, 
$v(n-1)=k(k-1)$ with $1<k=|\sD|< n-1$. In
$\PGL_{u+1}(\bF_q)$, $\alpha_n^u$ is conjugate to $\alpha_n$ if and only if $u\cdot
\sD$ is a translation of $\sD$. That is, $u$ is a {\sl multiplier\/} of the design. Also, -1
is always a nonmultiplier \cite[Lem.~5]{FrRedPol}.  Here $n=7$, so 
$v(n-1)=k(k-1)$ implies
$k=3$ and $v=1$. 
You find mod 7: 
$\sD=\{1,2,4\}$ and $-\sD$ are the only difference sets mod translation. The multiplier of $\sD$ is
$M_7=\lrang{2} \le (\bZ/7)^*$: $2\cdot \sD=\sD+c$ ($c=0$ here). 

Covers in this Nielsen class have genus 0. Now use that in $\PSL_3(\bZ/2)$,
(acting on points) the minimal possible index is 2. We labeled the transvection conjugacy class 
achieving that as $\C_{2^2}$.  So, 2 is the index of entries of $\bg$ for all finite branch
points. We have shown \eql{propDiff}{propDiffb} has the only
two possible Nielsen classes. 

\subsection{Completing property \eql{propDiff}{propDiffb}} We now show the two
spaces $\sH_{\pm}$ are irreducible, completing property \eql{propDiff}{propDiffb}.  

Computations in \cite[p.~349]{Fr95} list absolute Nielsen class representatives with 
$g_\infty$ the 4th entry. Label finite branch cycles $(g_1,g_2,g_3)$ (corresponding
to a polynomial cover, having $g_\infty$ in fourth position) as $Y_1,\dots,Y_7$. There
are 7 up to conjugation by $\lrang{g_\infty}$, the only allowance left for
absolute equivalence. \begin{equation}\begin{array}{rlrl} \label{7dav}
Y_1:& \!\!\!\!\!\!((3\,5)(6\,7),((4\,5)(6\,2),(3\,6)(1\,2));&  
\!\!\!\!\!Y_2:&\!\!\!\!\!\!((3\, 5)(6\, 7),(3\, 6)(1\, 2),(3\, 1)(4\, 5)); \\
Y_3:& \!\!\!\!\!((3\, 5)(6\, 7),(1\, 6)(2\, 3), (4\, 5)(6\, 2));&  
\!\!\!\!\!\!Y_4:&\!\!\!\!\!((3\, 5)(6\, 7),(1\, 3)(4\, 5),(2\, 3)(1\, 6));\\
Y_5:& \!\!\!\!\!\!((3\, 7)(5\, 6),(1\, 3)(4\, 5),(2\, 3)(4\, 7));& 
\!\!\!\!\!Y_6:&\!\!\!\!\!\!((3\, 7)(5\, 6),(2\, 3)(4\, 7),(1\, 2)(7\, 5));\\
Y_7:& \!\!\!\!\!\!((3\, 7)(5\, 6),((1\, 2)(7\, 5),(1\, 3)(4\, 5)).&&
\end{array}\end{equation}

The element $(3\,5)(6\,7)$ represents a transvection fixing  points of a line
$\Longleftrightarrow$ elements of $\sD$. Note: All entries in $\row Y 7$ of Table
\eqref{7dav}  correspond to transvections. So these are conjugate to $(3\,5)(6\,7)$. 
From this point everything reverts to Hurwitz monodromy calculation
with The elements
$q_i$, $i=1,2,3$. Each acts by a twisting action on any 4-tuple representing a Nielsen
class element. For example, 
\begin{equation} \label{twisting} q_2: \bg\mapsto (\bg)q_2=(g_1,g_2g_3g_2^{-1}, g_2,g_4).
\end{equation}   

\subsection{The analog for $n=13$ and 15} \label{n=13} Up to translation there are 4 difference sets
modulo 13. All cases are similar.  So we choose $\sD_{13}=\{ 1, 2, 4, 10 \}$ to be specific. 
Others come  from multiplications by elements of $(\bZ/13)^*$. Multiplying by
$\lrang{3}=M_{13}$  (\S\ref{diffSets}) preserves this difference set (up to
translation). 

Each
of $g_1, g_2, g_3$ fixes all points of some line, and one extra point, a total of five
points.  Any column matrix $A=(\bv|\be_2|\be_2)$ with $\bv$ anything, and
$\{\be_i\}_{i=1}^3$ the standard basis of $\bF_3^3$, fixes all points of the plane
$P$ of vectors with 0 in the 1st position. Stipulate one other fixed point in
$\prP^2(\bF_3)$ to determine $A$ in $\PGL_3(\bF_3)$. 

Let  $\zeta_{13} =
e^{2\pi  i/13}$. Identify $G(\bQ(\zeta_{13})/\bQ)$ with
$(\bZ/(13))^*$.  Let $K_{13}$ be the fixed field
of $M_{13}$ in $\bQ(\zeta_{13})$.  Therefore $K_{13}$ is
$\bQ(\zeta_{13} + \zeta^3_{13} + \zeta^9_{13})$, a degree 4 extension of $\bQ$. Akin to
when $n=7$ take $g_\infty=(1\,2\,\dots \,12\,13)^{-1}$. The distinct
difference sets (inequivalent under translation) appear as $6^j\cdot \sD_{13}$, $j=0,1,2,3$
(6 generates the order 4 cyclic subgroup of $(\bZ/13)^*$). 

For future reference, though we don't do the case $n=15$ completely here, $\sD_{15}=\{0, 5, 7,
10,11,13,14\}$ is a difference set $\mod 15$. Its multiplicity is $v$ in $v(15-1)=k(k-1)$ forcing
$k=7$ and $v=3$. The multiplier group is $M_{15}=\lrang{2}$, so the minimal field of
definition of polynomials in the corresponding Davenport pairs is
$\bQ(\sum_{j=0}^3 \zeta_{15}^{2^j})$, the degree 2 extension of $\bQ$ that $\frac{1+\sqrt{-15}}{2}$
generates. So, this case, like $n=7$, has two families of polynomials appear as
associated Davenport pairs.   

As with $n=7$, use the notation $\C_{(\alpha)^u}$, $u\in (\bZ/13)^*$ for the
conjugacy classes of powers of $13$-cycles. Prop.~\ref{H13comp} shows the following. 

\begin{edesc} \label{propDiff13} \item \label{propDiff13a} Why $\C_{2^{k_i}}=\C_{2^4}$,
$i=1,2,3$,  is the conjugacy class fixing all points of a plane (denote the resulting conjugacy
classes by $\bfC_{13,u;3\cdot 2^4}$).   
\item \label{propDiff13b} Why the four components of $\sH_{13}^\D$  are
$$\sH_{i}=\sH(\PGL_3(\bZ/3),\bfC_{13,6^i;3\cdot 2^4})^{\abs,\rd}, i=0,1,2,3.$$ 
\item \label{propDiff13c} Why  closure of all $\sH_{i}\,$s over $\prP^1_j$  are equivalent genus
0, degree 13 covers  over $\bQ$. \!Yet,  $\!\sH_{i}\!$, as degree \!13 Davenport moduli,  has
definition field $K_{13}$. 
\end{edesc} 

\section{$j$-line covers for polynomial $\PGL_3(\bZ/2)$ monodromy} \label{deg7j-line} 
We
produce branch cycles for the two $j$-line covers
$\bar \psi_{\pm}: \bar \sH_{\pm}\to \prP^1_j$, $i=1,2$, parametrizing degree 7 Davenport
(polynomial) pairs. Prop.~\ref{H7comp} says they are equivalent as
$j$-line covers, though distinct as families of degree 7 covers. 

\subsection{Branch cycle presentation and definition field} Our original notation,
$\sH_{\pm}$ is for points of
$\bar
\sH_{\pm}$ not lying over
$j=\infty$.  Each
$\bp_+\in \sH_{+}(\bar\bQ)$ has a corresponding point $\bp_-\in
\sH_{-}(\bar\bQ)$ denoting a collection of polynomial pairs $$\{(\beta\circ
f_{\bp_+},\beta\circ f_{\bp_-})\}_{\beta\in \PGL_2(\bar\bQ)}.$$  The absolute Galois group
of  
$\bQ(\bp_+)=\bQ(\bp_-)$ maps this set into itself. Representatives for absolute Nielsen
class elements in Table
\eqref{7dav} suffice for our calculation. This is because reduced equivalence adds the
action of 
$\sQ''=\lrang{(q_1q_2q_3)^2, q_1q_3^{-1}}$. This has the following effect. 
\begin{triv} \label{redLoc} Each $\bg\in \ni_+$ is reduced equivalent to a unique absolute Nielsen
class representative with $g_\infty$ in the 4th position. \end{triv} 
\noindent Example: If $\bg$ has $g_\infty$ in the 3rd position, apply $q_1^{-1}q_3$ to put
it in the 4th position. 

\cite[Prop.~4.4]{BFr} produces branch cycles for $\bar
\psi_{\pm}$. Reminder: The images of $\gamma_0=q_1q_2$ 
and $\gamma_1=q_1q_2q_3$ in $\lrang{q_1,q_2,q_3}/\sQ''=\PSL_2(\bZ)$ identify with canonical
generators of respective orders 3 and 2.  The product-one condition
$\gamma_0\gamma_1\gamma_\infty=1$ with 
$\gamma_\infty=q_2$ holds $\mod \sQ''$. Compute that  $q_1$ with 
\eqref{twisting} action is
$q_1^*=(3\,5\,1)(4\,7\,6\,2)$;  $q_2$ acts as 
$q_2^*=(1\,3\,4\,2)(5\,7\,6)$. So, $\gamma_0$ acts as
$\gamma_0^*=(3\,7\,5)(1\,4\,6)$. From product-one,  $\gamma_1$
acts as $\gamma_1^*=(3\,6)(7\,1)(4\,2)$. 

Denote the respective conjugacy classes of
$(\gamma_0^*,\gamma_1^*,\gamma_\infty^*)=\pmb\gamma^*$ in the group they generate by 
$\bfC^*=(\C_0,\C_1,\C_\infty)$. Denote $\prP^1_j\setminus
\{\infty\}$ by $U_\infty$. 
 
\begin{prop} \label{H7comp} The group $G=\lrang{\gamma_0^*,\gamma_1^*,\gamma_\infty^*}$ is 
$S_7$. Then, $\pmb\gamma^*$ represents the only element in
$\ni(S_7,\bfC^*)'$:  absolute equivalence classes with entries, in order, in the conjugacy class
$\bfC^*$. So, there is a unique cover $\bar\psi_7: X_7=X\to \prP^1_j$ in 
$\ni(S_7,\bfC^*)'$ (ramified over $\{0,1,\infty\}$). Restricting over $U_\infty$  gives 
$\psi_7:X^\infty \to U_\infty$ equivalent to      
$\sH_{\pm}\to U_\infty$ of \eql{propDiff}{propDiffb}: It parametrizes each of the two
absolute  reduced families of $\prP^1_z$ covers representing  degree 7 polynomials that
appear in a DP. 

The projective curve $X$ has  genus 0. It is not
a modular curve. The spaces $\sH_{\pm}$ are b-fine (but not fine; \S\ref{bfineMod})
moduli spaces over
$\bQ(\sqrt{-7})$. Their corresponding Hurwitz spaces are fine moduli spaces with a dense
set of
$\bQ(\sqrt{-7})$ points.  As a cover, however,  $\bar\psi_7$ has definition field  $\bQ$,
and $X$ has a dense set of $\bQ$ points.  
\end{prop}

\begin{proof} Since the group $G$ is transitive of degree 7, it is automatically primitive. 
Further, $(\gamma_\infty^*)^4$ is a 3-cycle. It is well-known that a primitive subgroup
of $S_n$ containing a 3-cycle is either $A_n$ or $S_n$. In, however, our case
$\gamma_\infty^*\not\in A_7$, so
$G=S_7$. 

We outline why   $\ni(S_7,\bfC^*)'$ has but one element. The centralizer of $\gamma^*_\infty$ is
$U=\lrang{(1\,3\,4\,2),(5\,7\,6)}$. Modulo absolute equivalence, any 
$(g_0,g_1,g_\infty)\in \ni(S_7,\bfC^*)'$ has $\gamma_\infty^*$ in the 3rd position. Let
$F=\{2,5\}$ and let $x_i$ be the fixed element of $g_i$, $i=0,1$. Elements of $F$ represent the
two orbits
$O_1,O_2$ of
$\gamma_\infty^*$ ($2\in O_1$ and $5\in O_2$). Conjugating by elements of $U$ gives four 
possibilities:
\begin{edesc} \label{onetwo} \item \label{onetwoa} (*) $x_0=2$ and $x_1\in O_1\setminus
\{2\}$; or (**)
$x_0=5$ and
$x_1\in O_2\setminus \{5\}$; or   
\item \label{onetwob} (*) $x_0=5$ and $x_1=2$; or (**) $x_0=2$ and $x_1=5$. \end{edesc}
 
We show the only possibility is \eql{onetwo}{onetwob} (**), and for that,  there
is but one element. First we eliminate \eql{onetwo}{onetwoa} (*) and (**). For 
\eql{onetwo}{onetwoa} (**), then $x_0=5$ and $x_1=7$ or 6. The former forces (up to
conjugation by $U$) 
$$g_0=(7\,6\,1)\cdots,g_1=(5\,6)\cdots.$$ Then, $g_1\gamma_\infty^*$ fixes 6, 
contradicting $(6)g_0=1$. Also, $x_1=6$ fails. Consider
\eql{onetwo}{onetwoa} (*), so $x_1=1$ or 3 ($(2\,4)$ appears in $g_0$ automatically). Symmetry
between the cases allows showing only $x_1=1$. This forces 
$(1\,4\,3)$ in $g_0$ and $\lrang{\gamma_\infty^*,g_0}$ is not transitive. 

Now we eliminate \eql{onetwo}{onetwob} (*). Suppose $x_0=5$, $x_1=2$. 
This forces either  $$g_0=(6\,1\,2)\cdots,\ g_1=(5\,6)(1\,7)\cdots \text{ or }
g_0=(6\,x\,?)(1\,2\,y),\ g_1=(5\,6)(x\,7)(y\,4)\cdots.
$$ In the first,  $4\mapsto 2 \mapsto 6 \to 5$ in the product, so it doesn't
work.  By inspection, no value works for $y$ in the latter. 

Conclude $x_0=2$, $x_1=5$. Previous analysis now produces but one 
possible element in $\ni(S_7,\bfC^*)'$. Applying Riemann-Hurwitz, using the index
contributions of $g_0^*,g_1^*,g_\infty^*$ in order,  the
genus of the cover as $g_7$ satisfies $2(7+g_7-1)=3+4+5$. So, $g_7=0$. 

That $X$  is not a modular curve follows
from WohlFahrt's's Theorem \cite{Wohlfahrt}. If it were, then its geometric monodromy group
would be a quotient of $\PSL_2(\bZ/N)$: $N$ is the least common multiple of the cusp
widths. So, $\ord(\gamma_\infty^*)=N=12$. Since, however,
$\PSL_2(\bZ/12)=\PSL_2(3)\times
\PSL_2(4)$, it does not have $S_7$ as a quotient.  

We account for the b-fine moduli
property. This is equivalent to $\sQ''$ acting faithfully on absolute Nielsen
classes \cite[Prop.~4.7]{BFr}. This is so from \eqref{redLoc}, 
nontrivial elements of
$\sQ''$ move the conjugacy class of the 7-cycle. \cite[Prop.~4.7]{BFr}
also shows it is a fine reduced moduli space if and only if  
$\gamma_0^*$ and $\gamma_1^*$ have no fixed points. Here   both have 
fixed points. So, the spaces $\sH_{\pm}$ are not fine moduli spaces. \end{proof}

\subsection{b-fine moduli property} \label{bfineMod} We explain \cite[\S4.3.1]{BFr} for our 
Davenport polynomial spaces. For the arithmetician, $\sP$ having {\sl fine\/} moduli over a field
$K$ has this  effect. For 
$\bp\in \sP$ corresponding to a specific algebraic object up to isomorphism, you have a 
representing  object with equations over $K(\bp)$ with $K$. (We 
tacitly assume  $\sP$ is quasiprojective \wsp our $\sH$ and $\sH^{\rd}$ spaces are
actually affine \wsp to give meaning to the field generated by the coordinates of a point.) 

\subsubsection{The meaning of b-fine} \label{bfine} The b in b-fine stands for {\sl birational}. It
means that if 
$\bp\in
\sH^{\rd}$ is not over
$j=0$ or 1, the  interpretation above for $\bp$ as a fine moduli point applies. It
{\sl may\/} apply if $j=0$ or 1, though we cannot guarantee it. 

If $\sP$ is only a moduli space (not fine), then $\bp\in \sP$ may have no representing
object over 
$K(\bp)$ (and certainly can't have one over a proper subfield of $K(\bp)$). Still,
$G_{K(\bp)}$ stabilizes the complete set of objects over  $\overline {K(\bp)}$ representing
$\bp$. 

For any Nielsen class of four branch point covers, suppose the  absolute (not reduced) space
$\sH$ has fine moduli. The
condition for that is no element of
$S_n$ centralizes
$G$. That holds automatically for any primitive non-cyclic group $G\le S_n$ (so in our
cases; all references on rigidity in any form have this). For $\bp$, denote the set of
points described as the image of  $\bp$ in
$U_4=((\prP_z^1)^4\setminus
\Delta_4)/S_4$ by $\bz_\bp$. Assume $K$ is the field of rationality of the conjugacy
classes. Conclude: There is a representative cover
$\phi_\bp: X_\bp\to
\prP^1_z$ branched over $\bz_\bp$ and having definition field $K(\bp)$. 

\subsubsection{Producing covers from the b-fine moduli property} Use the notation $\sH^\abs$ for
nonreduced space representing points of the Nielsen class. Then $\sH_{\pm}$ identifies with
$\sH^\abs_{\pm}/\PGL_2(\bC)$. Consider our degree 7 Davenport pair problem and their reduced
spaces.  The Nielsen classes are rational over 
$K=\bQ(\sqrt{-7})$ (Prop.~\ref{H7comp}). Any 
$\bp^\rd\in
\sH_{\pm}$ represents a  cover $\phi_{\bp^\rd}: X_{\bp^\rd}\to Y_{\bp^\rd}$ with
$Y_{\bp^\rd}$ a conic in
$\prP^2$ \cite[Prop.~4.7]{BFr} and $\phi_{\bp^\rd}$  degree 7. Often, even for  b-fine moduli and
$\bp^\rd$ not over 0 or 1, there may be no $\bp\in \sH^\abs$ lying over
$\bp^\rd$ with
$K(\bp)=K(\bp^\rd)$. Such a $\bp$ would give   
 $\psi_\bp: X_\bp \to \prP^1_z$ over $K(\bp^\rd)$ representing $\phi_{\bp^\rd}$. 

Yet, in our special case,  $X_{\bp^\rd}$ has a unique point totally ramified over
$Y_{\bp^\rd}$. Its image in $Y_{\bp^\rd}$ is  a $K(\bp^\rd)$ rational point. So
$Y_{\bp^\rd}$ is isomorphic to $\prP^1_z$ over $K(\bp^\rd)$. 

Up to an affine change of variable over
$K$, there is a copy of
$\prP^1=X$ over
$K$ that parametrizes degree 7 Davenport  pairs (over any nontrivial extension $L/K$). 

\section{$j$-line covers for polynomial $\PGL_3(\bZ/3)$ monodromy} \label{deg13j-line} We start by
listing the 3-tuples
$(g_1,g_2,g_3)$ for $X_i$ in Table
\eqref{13dav} that represent an absolute Nielsen class representative by tacking
$g_\infty=(1\,\cdots\,13)^{-1}$ on the end \cite[\S8]{Fr-Schconf}. The rubric proceeds just as in
the degree 7 case. 
\subsection{Degree 13 Davenport branch cycles $(g_1,g_2,g_3,g_\infty)$} 
The following elements are involutions fixing the 
hyperplane corresponding to the difference set. Each  fixes 
one of the 9 points off the hyperplane. 
Compute directly possibilities for
$g_1, g_2, g_3$ since each is a conjugate from this list:
$$\begin{array}{lll}
(7\, 8) (5\, 11)(6\, 12)(9\, 13);&\ (3\, 11)(7\, 13)(6\, 8)(9\, 12);&\ 
(3\, 12)(5\,8)(7\, 9)(11\, 13);\\ (5\, 13)(6\, 9)(11\, 12)(3\, 8);&\ 
(5\, 6)(3\, 7)(8\, 11)(12\, 13);&\ (6\, 7)(8\, 11)(5\, 12)(3\, 13);\\
(3\, 5)(7\, 12)(6\, 13)(8\, 9);&\  (3\, 6)(5\, 9)(7\, 11)(8\, 13);&\ 
(5\, 7)(6\, 11)(8\, 13)(3\, 9).\end{array}$$ 
Here is the table for applying the
action of
$q_1,q_2,q_3$. 
 
\begin{equation}\begin{array}{rlll} \label{13dav}
X_1: \!\!\!\; & \!\!\!\!\!(6\, 7)(8\, 11)(5\, 12)(3\, 13),\ & \!\!\!\!\! (2\, 3)(13\, 4)(6\,
8)(9\, 10),\  & \!\!\!\!\! (1\, 2)(13\, 5)(6\, 12)(9\, 11) \\
X_2:  \!\!\!\; & \!\!\!\!\!(6\, 7)(8\, 11)(5\, 12)(3\, 13),\ & \!\!\!\!\!  (1\, 2)(13\, 5)(6\,
12)(9\, 11),\  & \!\!\!\!\!  (1\, 3)(5\, 4)(12\, 8)(11\, 10)\\
 X_3:  \!\!\! \; & \!\!\!\!\! (3\, 5)(7\, 12)(6\, 13)(8\, 9),\ & \!\!\!\!\!  (1\, 6)(2\, 3)(13\,
7)(12\, 10),\  & \!\!\!\!\!  (8\, 10)(12\, 11)(6\, 2)(5\, 4)\\
X_4:  \!\!\! \; & \!\!\!\!\!  (3\, 5)(7\, 12)(6\, 13)(8\, 9),\ & \!\!\!\!\!  (8\, 10)(12\,
11)(6\, 2)(5\, 4),\  & \!\!\!\!\! (1\, 2)(6\, 3)(13\, 7)(11\, 8) \\
X_5:  \!\!\! \; & \!\!\!\!\!  (5\, 6)(3\, 7)(9\, 11)(12\, 13),\ & \!\!\!\!\!  (1\, 3)(4\, 5)(8\,
12)(11\, 10),\  & \!\!\!\!\! (2\, 3)(7\, 4)(1\, 8)(12\, 9) \\
X_6:  \!\!\! \; & \!\!\!\!\!  (5\, 6)(3\, 7)(9\, 11)(12\, 13),\ & \!\!\!\!\! (9\, 12)(2\, 3)(7\,
4)(1\, 8),\  & \!\!\!\!\! (8\, 2)(7\, 5)(1\, 9)(11\, 10) \\
X_7: \!\!\! \; & \!\!\!\!\! (5\, 6)(3\, 7)(9\, 11)(12\, 13),\ & \!\!\!\!\! (1\, 9)(2\, 8)(7\,
5)(10\, 11),\  & \!\!\!\!\! (4\, 5)(3\, 8)(9\, 2)(12\, 1)\\
X_8:  \!\!\! \; & \!\!\!\!\! (5\, 6)(3\, 7)(9\, 11)(12\, 13),\ & \!\!\!\!\! (4\, 5)(3\, 8)(9\,
2)(12\, 1),\  & \!\!\!\!\! (12\, 2)(9\, 3)(7\, 4)(11\, 10)\\
X_9: \!\!\! \; & \!\!\!\!\!(8\, 9)(7\, 12)(13\, 6)(3\, 5),\ & \!\!\!\!\!  (1\, 2)(6\, 3)(13\,
7)(11\, 8),\  & \!\!\!\!\! (1\, 3)(5\, 4)(12\, 8)(11\, 10)\\
X_{10}: \!\!\! \; & \!\!\!\!\!  (6\, 7)(8\, 11)(5\, 12)(3\, 13),\ & \!\!\!\!\!  (1\, 3)(4\,
5)(12\, 8)(11\, 10),\  & \!\!\!\!\! (6\, 8)(10\, 9)(4\, 13)(3\, 2) \\
X_{11}: \!\!\! \; & \!\!\!\!\!  (7\, 3)(5\, 6)(12\, 13)(9\, 11),\ & \!\!\!\!\!  (1\, 2)(7\, 5)(3\,
12)(8\, 9),\  & \!\!\!\!\! (1\, 3)(5\, 4)(12\, 8)(11\, 10) \\
X_{12}:  \!\!\! \; & \!\!\!\!\!  (5\, 6)(3\, 7)(9\, 11)(12\, 13),\ & \!\!\!\!\!  (10\, 11)(2\,
12)(3\, 9)(7\, 4),\  & \!\!\!\!\!  (1\, 2)(3\, 12)(7\, 5)(9\, 8) \\
X_{13}: \!\!\! \; & \!\!\!\!\!  (8\, 9)(6\, 13)(7\, 12)(3\, 5),\ & \!\!\!\!\!  (1\, 3)(5\, 4)(12\,
8)(11\, 10),\  & \!\!\!\!\! (2\, 3)(1\, 6)(7\, 13)(10\, 12). \end{array}\end{equation}

Here are the $j$-line branch cycle descriptions. From  action \eqref{twisting} on Table
\eqref{13dav}
\begin{equation} \label{Dav13-cycles} 
\begin{array}{rl}\gamma_0^*=q_1q_2=&(1\,5\,3)(6\,9\,13)(2\,8\,11)(4\,7\,10)
\\
\gamma_1^*=q_1q_2q_3=&(1\,4)(2\,5)(3\,6)(7\,9)(8\,10)(11\,12),\\
\gamma_\infty^*=q_2=&(1\,10\,2)(3\,13\,9\,4)(5\,11\,12\,8\,7\,6).\end{array}\end{equation}

Again, you figure $\gamma_1^*$ from the product one condition. In analogy to
Prop.~\ref{H7comp} we prove properties \eqref{propDiff13} for degree 13 Davenport pairs. Denote 
$\lrang{\gamma_0^*,\gamma_1^*,\gamma_\infty^*}$ by $G$. 

\begin{prop} \label{H13comp} Then,  $G=A_{13}$. With $\bfC^*=(\C_0,\C_1,\C_\infty)$ the conjugacy
classes, respectively, of
$\pmb \gamma^*$, $\ni(A_{13},\bfC^*)'$ (absolute equivalence classes with entries in order in
$\bfC^*$) has one element. So, there is
a unique cover $\bar\psi_{13}: X=X_{13}\to \prP^1_j$ representing it (ramified over
$\{0,1,\infty\}$).  Restrict over
$U_\infty$ for $\psi_{13}:X^\infty \to U_\infty$ equivalent to each       
$\sH_{j}\to U_\infty$, $j=0,1,2,3$. 

The projective curve $X$ has  genus 0 and is not 
a modular curve. The spaces $\sH_{j}$ are b-fine (not fine)
moduli spaces over
$K_{13}$. Their corresponding Hurwitz spaces are fine moduli spaces with a dense
set of
$K_{13}$ points.  As a cover, however,  $\bar\psi_{13}$ has definition field  $\bQ$,
and $X$ has a dense set of $\bQ$ points. 
 \end{prop}

\subsection{Proof of Prop.~\ref{H13comp}} First compute the genus $g_{13}$ of the
curve in the cover presented by
$\bar
\psi_{13}$ to be 0, from  
$$2(13+g_{13}-1)=\ind(\gamma_0^*)+\ind(\gamma_1^*)+\ind(\gamma_\infty^*)=4\cdot
2+6+(2+3+5)=24.$$ 

Now we show why the geometric monodromy of the
cover $\bar\psi_{13}$ is $A_{13}$. There are nine primitive groups of degree 13
\cite[p.~165]{carmichael}. Three affine groups $\bZ/p\xs U$; each   $U\le
(\bZ/13)^*$, with $U=\{1\}$, or having order 2 or order 3. Then, there are six other 
groups
$S_{13},A_{13}$ and
$\PGL_{3}(\bZ/3),\PSL_3(\bZ/3)$ with each of the last two acting on points and hyperplanes. The
generators
$\gamma_0^*$ and
$\gamma_1^*$ are in
$A_{13}$.  It would be cute if the monodromy group were $\PGL_{3}(\bZ/3)$ (the same as of the
covers represented by its points).  We see, however, $(\gamma_\infty^*)^6$ fixes way more
than half the integers in $\{1,\dots,13\}$. This is contrary to the properties of
$\PSL_{3}(\bZ/3),\PGL_3(\bZ/3)$. 

Now consider $\ni(A_{13},\bfC^*)'$.  Like $n=7$, assume triples of form
$(g_0,g_1,\gamma_\infty^*)$.  To simplify,
conjugate by an $h\in S_{13}$ to change $\gamma_\infty^*$ to 
$(1\,2\,3\,4)(5\,6\,7\,8\,9\,10)(11\,12\,13)$ (keep its name the same). Its centralizer 
is   
$U=\lrang{(1\,2\,3\,4),(5\,6\,7\,8\,9\,10), (11\,12\,13\,14)}$.  Take 
$F=\{1,5,11\}$ as representatives, in order, of the three orbits $O_1,O_2,O_3$ of
$\gamma_\infty^*$ on $\{1,\dots, 13\}$. Suppose $g_i$ fixes $x_i\in \{1\,\dots,13\}$,
$i=0,1$. If $x_0$ and $x_1$ are in different $O_k\,$s,   
conjugate by 
$U$, and use transitivity of
$\lrang{g_0,g_1}$, to assume $x_0$ is one element of $F$ and $x_1$ another. We 
show the following hold. 

\begin{edesc} \label{fix} \item \label{fix1} Neither $g_0$ nor $g_1$ fixes an element of
$O_3$. 
\item \label{fix2} The fixed point of $g_0$ is in $O_1$ or $O_2$, and of $g_1$ in the
other. 
\item \label{fix3} With no loss we may assume either $x_0=5$ and $x_1=1$, or $x_0=1$ and
$x_1=5$, and the latter is not possible. 
\end{edesc}

\subsubsection{Proof of \eql{fix}{fix1}} Suppose $x_0\in O_3$; with no loss conjugating by
$U$ take it to be 11. Then, $g_0=(12\,13\,y)\cdots $ and $g_1=(11\,13)(y\,12)\cdots$
(using  product-one in the form $\gamma_\infty^* g_0g_1=1$). Then, however, $g_0g_1$
fixes $y$: a contradiction, since $\gamma_\infty^*$ fixes nothing. Now suppose $x_1\in
O_3$ ($x_1=11$). Then,
$g_0=(12\,11\,y)\cdots$ and
$g_1=(y\,13)\cdots$. Transitivity of $\lrang{\gamma_\infty^*,g_0}$ prevents $y=13$.
Conjugating by $U$ allows taking $y=1$ or $y=5$. If $y=1$, then you find 
$g_0=(12\,11\,1)(2\,13\,4)\cdots$ and $g_1=(1\,13)(4\,12)(2\,3)\cdots$. Now, 
$\lrang{\gamma_\infty^*,g_1}$ leaves $O_1\cup O_3$ stable, so is not transitive. 

\subsubsection{Proof of \eql{fix}{fix2}, case $x_0,x_1\in O_1$} With no loss $x_0=1$
and
$g_1=(1\,4)\cdots $. This forces $g_0=(2\,4\,z)\cdots$, so under our hypothesis,
$(3)g_0=3$. This forces $g_1=(1\,4)(2\,3)\cdots$ and $\lrang{\gamma_\infty^*,g_1}$ is
stable on $O_1$. 

\subsubsection{Proof of \eql{fix}{fix2}, case $x_0,x_1\in O_2$} With no
loss $$x_0=5\text{ and }g_1=(6\,10\,z)\cdots,g_1=(10\,5)(z\,9)\cdots.$$ We separately show
$x_1\in
\{6,7,8\}$ are impossible, and by symmetry, $x_1\in O_2$ is impossible. If $x_1=6$, then
$z=7$, and $g_1=(6\,10\,7)(8\,9\,w)\cdots,g_1=(10\,5)(7\,9)(w\,8)\cdots$. The
contradiction is that $g_0g_1$ fixes $w$. 

If $x_1=7$, then $g_0=(6\,10\,z)(8\,7\,w)\cdots,g_1=(10\,5)(z\,9)(w\,6)\cdots$. So neither
$z$ nor $w$ can be 8 or 9: With no loss $w\in \{1,11\}$. If $w=11$, then with no loss
$z\in
\{12,1\}$.  With $z=1$, $(4)g_0g_1=11$, a contradiction. If $z=12$, 
$$g_0=(6\,10\,12)(8\,7\,11)(13\,9\,u)\cdots,g_1=(10\,5)(12\,9)(11\,6)(8\,13)\cdots.$$
Check: $u$ must be 13, a contradiction. 

Finally, $w=1$ forces 
$g_0=(6\,10\,2)(8\,7\,1)(3\,9\,?)\cdots,g_1=(10\,5)(2\,9)(1\,6)(8\,4)\cdots$.  
So, $(9)g_0g_1=8$ forces $?=4$, and that forces $g_1$ to fix 3, a contradiction to $x_1=7$. 

That leaves $x_1=8$ and 
$g_0=(6\,10\,z)(9\,8\,w)\cdots,g_1=(10\,5)(z\,9)(w\,7)\cdots$. Conjugate by $U$ to 
assume $w\in \{1,11\}$. The case $w=1$ forces 
$$g_0=(6\,10\,4)(9\,8\,1)(2\,7\,3)\cdots,g_1=(10\,5)(4\,9)(1\,7)(6\,3)\cdots.$$ Conclude: 
$\lrang{g_0,\gamma_\infty^*}$ stabilizes $O_1\cup O_2$. If $w=11$, 
$g_0=(6\,10\,z)(9\,8\,11)\cdots,g_1=(10\,5)(z\,9)(11\,7)\cdots$. In turn, this 
forces
$z=10$ and 10 appears twice in $g_1$.  

\subsubsection{Proof of \eql{fix}{fix3}} From \eql{fix}{fix1} and \eql{fix}{fix2},
conjugate by $U$ for the first part of \eql{fix}{fix3}. We must show
$x_0=1, x_1=5$ is false. If this does hold, then 
$$g_0=(6\,5\,y)(2\,4\,z)(3\,w\,?)\cdots,g_1=(1\,4)(y\,10)(2\,w)(3\,z)\cdots.$$  
Note that $y\ne 2$, and 2 and 3 appear in distinct cycles in $g_1$
 using $\lrang{\gamma_\infty^*,g_1}$ is transitive. Here is the approach for the rest: Try each
case where $y,z,w$ are in $ O_3$. Whichever of $\{y,z,w\}$ we try, with no loss take this to be 11. 

Suppose $y=11$. Then, $g_1\gamma_\infty^*$ fixes 11, contrary to $g_0$ not fixing
it. Now suppose $z=11$. This forces $w=13$ and $\gamma_\infty^*g_0$ fixes 12, though
$g_1$ does not. Finally, suppose $w=11$, and get an analogous contradiction to that for
$y=11$. With $y,z,w\in O_2$ we have $(3\,w\,11)$ a 3-cycle of $g_0$. This forces $z=13$,
a contradiction. 

\subsubsection{Listing the cases with $x_0=5$, $x_1=1$} With these hypotheses: 
$$g_0=(2\,1\,y)(6\,10\,w)\cdots,g_1=(10\,5)(y\,4)(w\,9)\cdots.$$ We check that 
$y,w\not\in O_1$: $y=3$ (or $w=3$) to get simple
contradictions. Example: $w=3$ forces $(6\,2)$ in $g_1$; forcing $y=7$, and
$\lrang{g_0,\gamma_\infty^*}$ is not transitive. 
Now check that
$y\not \in O_3$, but
$w\in O_3$. Our normalization for being in
$O_3$ allows  $y=11$: 
$$g_0=(2\,1\,11)(6\,10\,w)(12\,4\,u)(3\,13\,v),g_1=(10\,5)(11\,4)(w\,9)(2\,13)\cdots.$$
So, $w$ is 7 or 8. The first forces $\gamma_\infty^*g_0$ to fix 6, the second forces
$g_1\gamma_\infty$ to fix 9.  

Now consider $y,w\in O_2$. If $y=7$ and $w=8$, then $g_1\gamma_\infty^* $ fixes 9. If 
$y=8$ and $w=7$, then $ \gamma_\infty^* g_0$ fixes 6.  We're almost done: Try $y\in O_2$, and
$w=11$. Then
$y=7$ or 8. Try
$y=7$:
$g_0=(2\,1\,7)(6\,10\,11)(3\,z\,?)\cdots,g_1=(10\,5)(7\,4)(11\,9)(z\,2)\cdots$. This
forces $z=6$, putting $z$ in $g_0$ twice. So, $y=8$: You find that from this start, 
$$g_0=(2\,1\,8)(6\,10\,11)(12\,9\,4)(13\,3\,7),g_1=(10\,5)(8\,4)(11\,9)(3\,12)(2\,7)(6\,13)$$
is forced, concluding that there is one element in $\ni(A_{13},\bfC^*)'$.  

It is easy that $A_{13}$ has no $\PSL_2(\bZ)$ quotient. The b-fine moduli
is  from $\sQ''$ acting faithfully on the location of the 13-cycle conjugacy class 
 (as  with $n=7$). That $X^\infty$ does not have fine moduli 
follows from  $\gamma_0^*$ and $\gamma_1^*$ having fixed points.

\newcommand{\Q}[1]{\text{Q}^{{#1}}}
\newcommand{\psigma}{{\pmb \sigma}}
\newcommand{\sN}{{\tsp N}}

\section{Projective systems of Nielsen classes} Let $F_2=\lrang{x_1,x_2}$ be the free
group on two generators. Consider two simple cases for a group $H$ acting
faithfully on
$F_2$,  $H=\bZ/2$ ($n=2$) and $H=\bZ/3$ ($n=3$). 
\begin{edesc} \label{Hnact} \item \label{Hnact2} The generator of $H_2$ acts as
$x_i\mapsto x_i^{-1}$,
$i=1,2$.
\item \label{Hnact3} The generator $1\in \bZ/3=H_3$ acts as $x_1\mapsto x_2^{-1}$ and
$x_2\mapsto x_1x_2^{-1}$.
\end{edesc}

We explain why these  cases contrast extremely in achievable Nielsen classes.  Let
$\bfC_{2^4}$ be four repetitions of the nontrivial conjugacy class of $H_2$. Similarly, 
$\bfC_{\pm 3^2}$ is two repetitions of each nontrivial $H_3$  class. Refer to $\bfC$ as 
$p'$ conjugacy classes if a prime $p$ divides the orders of no elements in $\bfC$. Example:
$\bfC_{\pm 3^2}$ are $2'$ classes. 

\subsection{Projective sequences of Nielsen classes} \label{projNCs} Assume
$G^*\to G$ is a group cover, with kernel a
$p$ group. Then $p'$ conjugacy classes lift uniquely to $G^*$ \cite[Part III]{FrMT}. This 
allows viewing $\bfC$ as conjugacy classes in 
appropriate covering groups. 

Let
$P$ be any set of primes. 
Denote the collection of finite $p$ group quotients of $F_2$, with
$p\not\in P$, by 
$\Q {F_2}(P)$. Denote those stable under  $H$ by $\Q {F_2}(P,H)$.
Consider inner Nielsen classes with some fixed $\bfC$, $P$
containing all primes dividing orders of elements in $\bfC$, and groups running over a
collection from $\Q {F_2}(P,H)$: 
$$\sN_H= \{\ni(G,\bfC)^\inn\}_{\{G=U\xs H\mid  U\in
\Q {F_2}(P,H)\}}.$$  Only  $P=P_n=\{n\}$ for
$n=2,3$ appear below, though we state a general problem. 

Suppose for some $p$, $\sG_{p,I}=\{U_i\}_{i\in I}$ is a projective subsequence of 
(distinct) $p$ groups from $\Q {F_2}(P,H)$. Form a limit group  $G_{p,I}=\lim_{\infty
\leftarrow i} U_i\xs H$. Assume further, all Nielsen classes
$\ni(U_i\xs H,\bfC)$ are nonempty. Then,  $\{\ni(U_i\xs H,\bfC)^\inn\}_{i\in I}$
forms a project system with a nonempty limit $\ni(G_{p,I},\bfC)$. 
If $\bfC$ has $r$ entries, then the Hurwitz monodromy group $H_r=\lrang{\row q {r-1}}$ naturally
acts (by
\eqref{twisting}) on any of these inner (or absolute) Nielsen classes. We use just $r=4$.  

\begin{prob} \label{projNiel} Assume $I$ is infinite. What are the maximal groups $G_{p,I}$ from
which we get nonempty limit Nielsen classes
$\ni(G_{p,I},\bfC)$? \end{prob} We call such maximal groups $\bfC$ $p$-Nielsen class limits.
Prob.~\ref{projNielCusp}  refines this.  

For $r=4$, consider maximal limits of
projective systems of reduced components that have genus 0 or 1. Allow $|I|$ here to be 
bounded. The strong Conjecture on Modular Towers
\cite[\S1.2]{schurtype} specializes to $n=3$ to say this. Each such sequence should be bounded and 
there should be only finitely many  (running over all
$p\not\in P_3$).  Such genus
0 or 1 components have application. Ex.~\ref{G1A5} is one such.  

Any profinite pro-$p$ group $\hat P'$ has a universal subgroup generated by $p$th
powers and commutators from $\hat P'$. This is the {\sl Frattini\/} subgroup, denoted
$\Phi(P')$. The $k$th iterate of this group is $\Phi^k(P')$. \S\ref{genus01Mod} referenced reduced
spaces
$\{\bar
\sH_{k,p}^\rd\}_{k\ge 0}$. These are the spaces for the Nielsen classes $\ni(\hat
F_{2,p}/\Phi^k(\hat F_{2,p})\xs H_3,\bfC_{\pm 3})^{\inn,\rd}$.  

\begin{exmp} \label{G1A5} Let
 $\phi_1: G_1(A_5)\to A_5$ be the universal exponent 2 extension of
$A_5$. We explain: If $\phi: G^*\to A_5$ is a cover with abelian exponent 2 group as kernel, then
there is a map $\psi: G_1(A_5)\to G^*$ with $\phi\circ \psi=\phi_1$. The
space $\sH_{2,1}^\rd$ has six components \cite[Ex.~9.1, Ex.~9.3]{BFr}.  Two have genus 0, and two have genus 1. 
Let $K$ be a real number field. Then, 
there is only one  possibility for infinitely many (reduced equivalence classes of) 
$K$ regular realizations of $G_1(A_5)$ with four branch points. It is
that the genus 1 components of $\bar \sH_{2,1}^\rd$ have definition field $K$ and  infinitely many
$K$ rational points. The genus 0 components here have no real points. \cite{schurtype} explains this
in more detail.  \end{exmp}

\subsection{$\sN_2\eqdef\{\ni(G,\bfC_{2^4})^\inn\}_{G\in \Q {F_2}(2)\xs H_2}$} 
The first sentence of Prop.~\ref{H2NC} restates an argument from 
\cite[\S1.A]{FrMT}. 

\subsubsection{Achievable Nielsen classes from modular curves} 
Let $\bz=\{\row z
4\}$ be any four distinct points of $\prP^1_z$, without concern to order. As in \S\ref{standEquiv} 
 choose a set of (four) classical generators for the fundamental group of
$\prP^1_z\setminus
\bz=U_\bz$. 

This group identifies with the free group on four generators $\psigma=(\row
\sigma 4)$, modulo the product-one relation $\sigma_1\sigma_2\sigma_3\sigma_4=1$. Denote
its completion with respect to all normal subgroups by
$\hat  F_\psigma$. Let $\hat \bZ_p$ (resp.~$\hat F_{2,p}$) be the similar completion of $\bZ$
(resp.~$F_2$) by all normal subgroups with $p$ group quotient. 
\begin{prop} \label{H2NC}  Let $\hat D_\psigma$ be
the quotient of $\hat  F_\psigma$ by the relations $$\sigma_i^2=1,\   i=1,2,3,4\ (\text{so\
}\sigma_1\sigma_2=\sigma_4\sigma_3).$$ Then, $\prod_{p\ne 2}\hat \bZ^2_p\xs H_2\equiv 
\hat D_\psigma$. Also, $\hat \bZ^2_p\xs H_2$ is the unique $\bfC_{2^4}$ $p$-Nielsen class limit. 
\end{prop} 

\begin{proof} We show combinatorially  that $\hat D_\psigma$ is  
$\hat\bZ^2\xs H_2$ and that  $\sigma_1\sigma_2$ and 
$\sigma_1\sigma_3$ are independent generators of $\hat\bZ^2$. Then,
$\sigma_1$ is a generator of $H_2$ which we regard as $\{\pm 1\}$ acting on $\hat\bZ^2$
by multiplication. First: 
$\sigma_1(\sigma_1\sigma_2)\sigma_1=\sigma_2\sigma_1$ shows
$\sigma_1$ conjugates $\sigma_1\sigma_2$ to its inverse. Also, 
$$(\sigma_1\sigma_2)(\sigma_1\sigma_3)= 
(\sigma_1\sigma_3)\sigma_3(\sigma_2\sigma_1)\sigma_3=
(\sigma_1\sigma_3)(\sigma_1\sigma_2)$$ 
shows the said generators commute. The maximal pro-$p$ quotient is 
$\bZ_p^2\xs\{\pm1\}$. 

We have only to show  Nielsen classes with $G=U\xs
H_2$, and $U$ an abelian quotient of $\bZ^2$, are nonempty. It suffices to deal with
the cofinal family of $U\,$s, $(\bZ/p^{k+1})^2$, $p\ne 2$. \S\ref{nonempNielsen2} has two proofs.
\end{proof}  

\begin{rem} Any line (pro-cyclic 
group) in $\bZ_p^2$ produces a {\sl 
dihedral group\/}  $D_{p^\infty}$ from multiplication by -1 on this 
line. \end{rem} 

\subsubsection{Nonempty Nielsen classes in Prop.~\ref{H2NC}}
\label{nonempNielsen2} In $G_{p^{k+1}}=(\bZ/p^{k+1})^2\xs \{\pm 1\}$, $\{(-1;\bv) \mid\bv\in
(\bZ/p^{k+1})^2\}$ are the involutions. Write   $\bv=(a,b)$,
$a,b\in \bZ/p^{k+1}$. The multiplication $(-1;\bv_1)(-1;\bv_2)$ yields $(1;\bv_1-\bv_2)$ as in 
the matrix product $$\smatrix {-1} {\bv_1} 0 1  \smatrix {-1}
{\bv_2} 0 1.$$ 

We have an explicit description of the Nielsen classes
$\ni(G_{p^{k+1}},\bfC_{2^4})$. Elements are 4-tuples $((-1;\bv_1),\dots,(-1;\bv_4))$
satisfying two conditions from \S\ref{nc1}:
\begin{edesc} \item Product-one: $\bv_1-\bv_2+\bv_3-\bv_4$; and
\item Generation:  $\lrang{\bv_i-\bv_j, 1\le i, j\le 4}=(\bZ/p^{k+1})^2$.
\end{edesc} Apply  conjugation in $G_{p^{k+1}}$ to assume $\bv_1=0$. Now take $\bv_2=(1,0)$, 
$\bv_3=(0,1)$ and solve for $\bv_4$ from the product-one.  This shows the Nielsen class is
nonempty. To simplify our discussion we have taken inner Nielsen classes. What really makes an
interesting story is the relation between inner and absolute Nielsen classes. Use the natural
inclusion $G_{p^{k+1}}\norm (\bZ/p^{k+1})^2\xs \GL_2(\bZ/p^{k+1})$  regarding both groups as
permutations of $(\bZ/p^{k+1})^2$.

A general theorem in \cite{FrVMS} applies here. It says the
natural map from $\sH(G_{p^{k+1}},\bfC_{2^4})^\inn\to
\sH(G_{p^{k+1}},\bfC_{2^4})^\abs$ is Galois with group
$\GL_2(\bZ/p^{k+1})/\{\pm1\}$. We give our second proof for nonempty Nielsen classes to clarify the
application of 
\cite[IV-20]{SeAbell-adic}. This shows, depending on the
$j$-value of the 4 branch points for the cover $\phi_\bp: X_\bp\to \prP^1_z$, we can say
explicit things about the fiber of
$\sH(G_{p^{k+1}},\bfC_{2^4})^\inn\to
\sH(G_{p^{k+1}},\bfC_{2^4})^\abs$ over $\bp\in \sH(G_{p^{k+1}},\bfC_{2^4})^\abs$.  

We will now display the cover $\phi_\bp$. Let $E$ be any elliptic curve in Weierstrass
normal form, and
$[p^{k+1}]: E \to E$ multiplication by $p^{k+1}$. Mod out by the action of $\{\pm 1\}$ on both
sides of this isogeny to get $$E/\{\pm 1\}=\prP^1_w\mapright{\phi_{p^{k+1}}} E/\{\pm
1\}=\prP^1_z,$$ a degree
$p^{2(k+1)}$ rational function. Compose  $E\to E/\{\pm 1\}$ and $\phi_{p^{k+1}}$  for the Galois
closure of $\phi_{p^{k+1}}$. This geometrically shows 
$\ni(G_{p^{k+1}},\bfC_{2^4})\ne \emptyset$. If $E$ has definition field $K$, so does
$\phi_{p^{k+1}}$. We may, however, expect the Galois closure field of $\phi_{p^{k+1}}$ to have 
interesting constants, from the  definition fields of $p^{k+1}$ division
points on
$E$. This is the subject of Serre's Theorem and 
\cite[\S 5.2-5.3]{FrExcTow}. 

\subsection{$\sN_3\eqdef \{\ni(G,\bfC_{\pm 3^2})^\inn\}_{G\in \Q {F_2}(3)\xs H_3}$} 
Prop.~\ref{H3NC} says, unlike  $n=2$,   for any 
$G\in \sN_3$, $\ni(G,\bfC_{\pm 3^2})^\inn$  is nonempty. This vastly differs from the conclusion of
Prop.~\ref{H2NC}. Our proof combines Harbater-Mumford reps. (Ex.~\ref{HMreps}) with the {\sl
cyclic\/}  action of $H_3$ on
$(\bZ/p)^2$.  We make essential use of the Frattini property.   
 
\begin{prop} \label{H3NC} $\hat F_{2,p}\xs H_3$ is the unique $\bfC_{\pm 3^2}$ $p$-Nielsen
class limit.  
\end{prop} 

\begin{proof} First we show Nielsen classes with $G=G_p=(\bZ/p)^2\xs H_3$ ($p\not=2$) are
nonempty by showing each contains an H-M rep.  Let $\lrang{\alpha}=\bZ/3$, with the
notation on the left meaning multiplicative notation (so, 1 is the identity). The action
of $\alpha$ is from  
\eql{Hnact}{Hnact3}. For multiplication in $G_p$ use the analog of that in
\S\ref{nonempNielsen2}. 

Suppose  $g_1=(\alpha,\bv_1), g_2=(\alpha,\bv_2)\in G$ generate $G$. Then,
$(g_1,g_1^{-1},g_2,g_2^{-1})$  is in $\ni(G,\bfC_{\pm 3^2})^\inn$. Further this element is an
H-M rep. Conjugate this 4-tuple by
$(1,\bv_3)$ (for inner equivalence) to change $g_1$ to
$(\alpha,\bv_1+\bv_3-\alpha(\bv_3))$. As 
$I_2-\alpha$ is invertible on $(\bZ/p)^2$ we can choose $\bv_3$ so $\bv_1={\pmb 0}$. To find such
generators, consider 
$g_1g_2^{-1}=(1,-\bv_2)$ and
$g_1^2g_2=(1,\alpha^{-1}(v_2))$. So, $g_1,g_2$ are generators precisely
if $\lrang{-\bv_2, \alpha^{-1}(v_2)}=(\bZ/p)^2$. Such a $\bv_2$ exists because the eigenvalues of
$\alpha$ are distinct. So $(\bZ/p)^2$ is a cyclic $\lrang{\alpha}$
module. If $x^2+x+1 \mod p$ has no solution,
then
$\lrang{\alpha}$ acts irreducibly and any 
$\bv_2\not={\pmb 0}$ works. From quadratic reciprocity this is equivalent to $-3$ is not a square
$\mod p$, or $p\equiv -1 \mod 3$.

Now consider any Nielsen class in $\sN_3$ defined by $G=U\xs H_3$ with $U$ having $(\bZ/p)^2$ as
a quotient. There is a surjective map
$\psi: G\to (\bZ/p)^2\xs H_3$, and it is a Frattini cover. So, if $g_1',g_2'$ are
generators of $(\bZ/p)^2\xs H_3$ given by the proof above, then any respective order 3
 lifts of $g_1',g_2'$ to
$g_1,g_2\in G$ will automatically generate $G$. Therefore the representative
$(g_1,g_1^{-1},g_2,g_2^{-1})$  of the Nielsen class $\ni(G,\bfC_{\pm 3^2})^\inn$ lifts
$(g_1',(g_1')^{-1},g_2',(g_2')^{-1})$. This shows all Nielsen classes are nonempty. 
\end{proof}

\subsection{Projective cusp types} \label{cusptype} Let $r=4$ and $\bg\in
\ni(G,\bfC)^{\inn,\rd}$ (or any reduced equivalence). Then, the orbit of $\bg$ under the cusp
group $\Cu_4\eqdef\lrang{q_2,\sQ''}$ interprets as a cusp of the corresponding $j$-line cover. 
Certain  Nielsen class representatives  define especially useful cusps. We define these relative to
a prime $p$. 
\cite[\S5.2]{schurtype} calls these g-$p'$ cusps, and when $r=4$, their representatives
$\bg=(\row g 4)$ in reduced Nielsen classes have the following property:   
\begin{triv} Both $H_{2,3}(\bg)\eqdef\lrang{g_2,g_3}$ and $H_{1,4}(\bg)=\lrang{g_1,g_4}$ are $p'$
groups; their orders are prime to $p$. \end{triv} 

Recall the shift $\sh=q_1\cdots q_{r-1}$ as an operator on Nielsen classes: For $$\bg\in
\ni(G,\bfC),\ (\bg)\sh=(g_2,\dots,g_r,g_1).$$  

\begin{exmp}[H-M reps.] \label{HMreps} Suppose $h_1$ and $h_2$ are elements generating a group
$G$ and consider the 4-tuple $\bh=(h_1,h_1^{-1},h_2,h_2^{-1})$. If $\bh\in \ni(G,\bfC)$, we say
$\bh$ is  a {\sl Harbater-Mumford\/} representative (H-M rep.) of the Nielsen class. If $\bfC$ are 
$p'$ conjugacy classes, then $(\bh)\sh$ is a 
a representative for a
$p'$ cusp. 
\end{exmp} 

\cite[Prop.~5.1]{schurtype} generalizes how we use H-M
reps., in Prop.~\ref{H3NC} as follows. Suppose level 0 of a Modular Tower, for the prime $p$, 
has representatives of g-$p'$ cusps. Then   applying Schur-Zassenhaus gives
projective systems of g-$p'$ cusps. This is the only method we know to show there are nonempty
Nielsen classes  at all levels. 

So, we don't know whether g-$p'$ cusps are necessary for a result
like Prop.~\ref{H3NC}. For this case alone we pose a problem that aims at deciding that.  Use
the action of $H_4=\lrang{q_1,q_2,q_3}$ in \S\ref{projNCs} (from \eqref{twisting}). For 
Nielsen classes from
$\sN_3$, it is easy to show all representatives of g-$p'$ cusps ($p\ne 3$) must be shifts of H-M
reps.

\begin{prob} \label{projNielCusp} Suppose $O$ is an $H_4$ orbit on $\ni(\hat
F_{2,p}\xs H_3,\bfC_{\pm 3})^{\inn,\rd}$. Must $O$ contain the shift of an H-M rep.?  
 \end{prob} 

\section{Relating \Prob1 and \Prob2} Our last topic discusses, using the work of others, what
might be the meaning and significance of an intrinsic uniformizer for a 1-dimensional genus 0 moduli
space. 

\!\!\!\subsection{Explicit equations}  \label{expEqs} As \S\ref{nc2} notes, our computations 
were with pencil  and paper. Many saw the then unpublished pieces we used from \cite[p.~349]{Fr95}
and
\cite[\S8]{Fr-Schconf}   in 
1969 at the writing of
\cite{FrRedPol}.  Precisely relating to the $j$-line is more recent, though you see these
computations are easy. Further, they create a framework for finding a parameter for a moduli cover
of the $j$-line having an internal interpretation. 

\subsubsection{Couveignes calculations} Jean-Marc Couveignes used computer-assisted calculations to
find equations for Davenport pairs \cite{couveignes2}. He also gave a technique for uniformizing
some moduli of the type we are considering. Necessarily, the moduli was of genus 0 covers,
and the moduli space of genus 0 \cite{couveignes1}.  Prop.~\ref{H7comp} and
Prop.~\ref{H13comp} produce  a natural geometry behind the genus 0 moduli from Davenport's
arithmetically defined problem. We compare our computation with that from
\cite{couveignes2}, whose tool was {\bf PARI Ver.~1.920.24} with {\bf Maple} used as a check. 

Equations for those degree 7, 13 and 15  polynomials are in 
\cite[\S5.1]{couveignes2}, \cite[\S5.3]{couveignes2} and \cite[\S5.4]{couveignes2}.
\cite[p.~593]{RETExp} has Birch's brute force calculation of degree 7 Davenport
pairs. He and Guy also did the degree 11 case (three branch point covers, so up
to reduced equivalence this is just one pair of polynomials). Degree 13 is the most interesting.
There the difference set argument produces the nontrivial intertwining of four polynomials for each
point of the space
$X_{13}^\infty$.  As the  \cite{couveignes2} calculations take
considerable space, we don't repeat them. Still, a statistical comparison 
indicates the extra complexity (supported by theory) in the degree 13 case.  Here
is the character count for  Couveignes' expression for the 
general polynomial in each case (not counting spaces): 
\begin{triv} degree 7: 146 ; degree 13: 1346 ;  degree 15: 819 .
\end{triv} 

\subsubsection{Significance of the $j$-line cover} 
Neither Birch nor Couveignes relates  their equations to the $j$-line. Still, two points
help compare \cite{couveignes1} with  our goals. 

\begin{edesc} \label{couvPts} \item \label{couvPtsa} Using \cite{GHP}, Harbater patching 
can sometimes produce the equation for the general member a family of genus 0
covers. 
\item \label{couvPtsb} Using formal fibers of a moduli problem avoids direct computation of a 
possibly large base extension  (when no version of rigidity holds). 
\end {edesc} I think \eql{couvPts}{couvPtsa} refers to moduli of genus 0
covers, when a Nielsen class gives data to normalize a 
parameter for each cover in the family. We explain below this computational
handle that starts from degenerate situations and at the cusps. 

Grothendiecks' original method,
reviewed in \cite[\S 3.6]{Fr-Schconf},  is totally compatible. Except,  applying
Grothendieck's method requires a computationally inexact {\sl Artin approximation\/} to achieve an
algebraic deformation. It is around that last step that Couveignes uses his  
genus 0 assumptions. 

The example of \cite[p.~48]{couveignes2}  has some version of  reduced parameters, though over the $\lambda$-line, not the $j$-line. \cite[p.~56]{couveignes2}
gives the practical sense of \eql{couvPts}{couvPtsb}. We use his formal deformation
parameter $\mu$.  His example serves no exterior problem.  Still, it illustrates the
computational technique. Though we admire it, we want to show why there are
many constraints on its use to  compute equations.  

\subsubsection{Constraints in Couveignes' example} \label{couveignesEx}
The group for his Nielsen classes is $S_7$, and his conjugacy classes $\bfC^c$ are those represented
by the entries of 
\begin{equation} \label{muc} \bg_\mu=((2\,3\,4)(6\,7),(1\,2\,5\,6),(1\,7),(2\,3\,4\,5\,6 \,
7)^{-1}).\end{equation} He interpolates between two 3 branch point covers with respective branch
cycles: 
\begin{equation} \label{mu01c} \begin{array}{rl}
\bg_0=&((2\,3\,4)(6\,7),(1\,2\,5\,6),(1\,2\,3\,4\,5\,6\, 7)^{-1})\text{ and; }\\
\bg_1=&((1\,2\,3\,4\,5\,6\,7), (1\,7),(2\,3\,4\,5\,6\, 7)^{-1}).\end{array} \end{equation}
His four conjugacy classes are all distinct. Then, the reduced family
of absolute covers in the Nielsen class, as a $j$-line cover,  automatically factors through
the $\lambda$-line as $\bar\psi^c: \bar \sH(S_7,\bfC^c)^{\abs,\rd}\to \prP^1_\lambda$. He finds
explicit equations $\phi_0$ and $\phi_1$ for the {\sl unique\/} covers branched over
$\{0,1,\infty\}\in \prP^1_j$ with branch cycles represented by $\bg_0$ and $\bg_1$. 
An analyst would view this as forming $\phi_\mu: \prP^1_{w_\mu}\to \prP^1_z$ as a function
of
$\mu\in (0,1)$ so that as $\mu\mapsto 0$ (resp.~1) $\phi_\mu$ degenerates to $\phi_0$
(resp.~$\phi_1$). 

Topologically this is a {\sl coalescing\/}, respectively, of the first two
(last two) branch points. Realize, however, compatible with the statement in \S\ref{standEquiv},
this is not an algebraic process. We want equations for $\phi_\mu$ in the coordinates of the
parameter space
$\bar \sH(S_7,\bfC^c)^{\abs,\rd}$. \cite[p.~43-48]{couveignes2} refers to  \cite[Thm.~4.5]{DFr} to compute the action of the
$\lambda$-line (not $j$-line) version of $\gamma_0^*,\gamma_1^*,\gamma_\infty^*$ (as in
Propositions~\ref{H7comp} and \ref{H13comp}). He
concludes 
$\bar
\sH(S_7,\bfC^c)^{\abs,\rd}$ has genus 0. Dropping reduced equivalence gives a fine moduli
space (reason as in \S\ref{bfine}).  

To track the $\lambda$-line, we must express coordinates for $\bp\in
\sH(S_7,\bfC^c)^{\abs,\rd}$    in the algebraic closure of $\bQ(\lambda)$.  This   
space is irreducible, from the analog of $\gamma_0,\gamma_1,\gamma_\infty$
acting transitively. For the same reason as in our Davenport pair examples, reduced equivalence gives a 
b-fine moduli space (\S\ref{bfineMod}). These are hypotheses that satisfy the easiest case of {\sl
braid rigidity\/}. So, the cardinality of the  reduced Nielsen class equals the degree of those
coordinates over
$\bQ(\lambda)$ (\cite[Cor.~5.3]{FrHFGG}, \cite[Thm.~5.3]{MM} or \cite[\S10.3.1]{VB}; responding to
the potential problem of \eql{couvPts}{couvPtsb}). 

\cite[p.~50]{couveignes1} uses a set of classical generators (as in \S\ref{standEquiv}) around
branch points. Then, the effect of coalescing a pair of branch points, say, as $\mu\mapsto 1$
interprets simply.  It is as if you replaced the 3rd and 4th of the classical generators by
their product. Still, this is topological, not algebraic, data. 

For that he  normalizes a
parameter
$w_\mu$ (to appear in $\phi_\mu(w_\mu)$) by selecting 3 distinguished points on the
cover $\phi_\mu$. These correspond to selecting 3 distinguished disjoint cycles in the branch
cycles
$\bg_\mu$ in \eqref{muc}. He wants  
$w_{\mu}=w_{\mu,1}$ to survive (in the limit 
$\mu \mapsto 1$) to give $w_1$ (similarly with $w_0$). So, the points he chooses must survive
the coalescing. It is a constraint on the explicit equations that he can pick three such
points.  It is a separate constraint that he can do the same on the other limit
$\mu\mapsto 0$, producing a parameter $w_{\mu,0}$ for $\mu$ near 0.  

Here in \cite[p.~50-51]{couveignes1} a reader might have difficulty (see the
reference).  Naming the three points (labeled $V_3,V_2,V_6$), where $w_{\mu,1}$ takes respective
values $0,1,\infty$ (chosen for
$\mu\mapsto 1$)  appears after their first use in equations. That is because of a misordering of
the printed pages.  He analytically continues these choices 
along $\mu\in (0,1)$. They assure 
$w_{\mu,0}/w_{\mu,1}$ is $w_{\mu,0}(V_2)/w_{\mu,1}(V_2)$, expressed in meaningful
constants from ramification of $\bar\psi^c$, as a function of $\mu$. 
\cite{couveignes1} then explains what to do with expressions for $\phi_\mu(w_{\mu,k})$, $k=0,1$
expressed as local power series. He uses their truncation up to the necessary degree of accuracy for
their determination from the algebraic conditions. Details  on the genus 0 moduli space come
into play to {\sl precisely\/} express coefficients of a general member of the family.  

\subsubsection{Using geometric compactifications} Couveignes applied  coordinates from \cite{GHP}
for his explicit  compactification. Many use a compactification around cusps, 
placing over the cusp something called an {\sl admissible\/} cover of (singular) curves. Arithmetic
applications require knowing that the constant field's absolute Galois group 
detects the situation's geometry. My version is a {\sl specialization
sequence\/}
\cite[Thm.~3.21]{FrMT}. This gives meaningful  action on projective
sequences of cusps. The goal is to see exactly how $G_\bQ$ acts 
from its preserving geometric collections like g-$p'$ cusps (\S\ref{projNCs}). 

My treatment, however, did not aim at explicit equations. 
\cite[Thm.~1.3 and Thm.~4.1]{DDE} gave a treatment of \cite[Thm.~3.21]{FrMT} using {\sl admissible
covers\/} (the version in
\cite{Wewers}). For that reason, they compactify with a family that has {\sl H-M admissible\/}
covers around the H-M cusps. (They use Hurwitz, not reduced Hurwitz, spaces.) 
A corollary of \cite[Thm.~3.21]{FrMT} has simple testable hypotheses that guarantee there is a
unique component of the moduli space containing H-M cusps (so it is over $\bQ$). Those
hypotheses rarely hold if $r=4$.  

So, I wish someone could do the following. 

\begin{prob} Approach the genus 1 components in Ex.~\ref{G1A5} as did Couveignes for his
examples. \end{prob}  Couveignes' explicit constrains  fail miserably here.
Ex.~\ref{G1A5} is a family of very high genus curves with no distinguished disjoint cycles in their
branch cycle descriptions.  Still, the topics of
\cite{DDE} and
\cite{Wewers} are relevant.  The two genus 1 components of Ex.~\ref{G1A5} are {\sl both\/} H-M
rep.~components. The two components come from corresponding orbits for $H_4$ acting on inner
Nielsen class orbits. They are, however, not total mysteries: An outer automorphism of
$G_1(A_5)$ joins the orbits. (None of that is obvious; \cite[\S7]{schurtype} will have complete
documentation.) 

\subsubsection{An intrinsic uniformizing parameter}  
Interest in variables separated polynomials $f(x) -g(y)$, those  \S\ref{DPs} calls Davenport
pairs, first came from factorization questions (\cite{DLSf-g} and
\cite{DSf-g}). As \cite[\S 1.1.2]{FrExcTow} explains,
Davenport's own questions showed his interest in the finite field properties. That
exhibited them  with delicate arithmetic properties.
Given that, we should express a uniformizing parameter, for $X_n$ ($n=7,13,15$) 
using the arithmetic behind their investigation. We don't know if there is such, though
\cite{FrExcTow} suggests some based on functions in $(x,y)$ (satisfying $f(x) -g(y)=0$) that are
constant along fibers to the space $X_n$. 

\subsection{A piece of John Thompson's influence} \label{Thomp}  I speak only of John's influence
in specific mathematical situations related to this paper. John and I had a conversation
in 1986 on the way to lunch at University of Florida. This one conversation brings a {\sl
luminous\/} memory, so singular it may compress many conversations. 

\subsubsection{John's formation of \Prob2} As we walked, I summarized the group theory  from
solving  several problems, like Davenport's. Each came with an    
an equation  that we could rewrite as a phrase on  primitive (genus 0)
covers. My conclusion: There were always but finitely many rational function degree
counterexamples to the most optimistic hopes. Yet,  there were some counterexamples. 

Further, to solve these problems required nonobvious aspects of groups. For
example: In Davenport's problem, we needed difference sets and knowledge of all the  finite groups
with two distinct doubly transitive permutation  representations that were equivalent as group
representations (I got this from \cite{CuKanSe}). Quite like the classification, many 
simple groups and some new number theory,  impinged on locating the polynomial Davenport pairs.
That was my pitch. 

John responded that he was  {\sl seized\/} with the underlying 
problem. His initial formulation was this. Suppose $G$ is a {\sl composition\/} factor of the
covering monodromy of $\phi:X\to \prP^1_z$ with $X$ of  genus
$g$. Then, it is a genus $g$ group. 

\begin{prob} We fix an integer $g$ and   
exclude alternating and cyclic groups. Show:  Only finitely many simple 
groups have genus $g$.  
\end{prob} 

As above, the genus 0 primitive cover case suffices. 
This was my encouragement for the project. It is 
nontrivial to {\sl grab\/} a significant  monodromy group at 
random. (You will always get the excluded groups.) Still, the genus 0 problem
would display exceptions. These should contribute conspicuously, as happened with the Schur,
Davenport and  HIT problems. That is, a general theorem would have sporadic 
counterexamples. While they might be baffling, they would nevertheless 
add to the perceived depth of the result. Especially they would guide
situations of higher genus, and in positive characteristic.  Example: I suggested there would be new
primitive rational functions  beyond those coming from elliptic curves, that had the {\sl Schur
cover  property}: Giving one-one maps on $\prP^1_z(\bF_p)$ for $\infty$-ly many primes $p$. 

John suggested we work toward this immediately. I had hoped, then, he would be interested in my
approach to using the universal $p$-Frattini cover of a finite group. My response was that Bob
Guralnick  enjoyed this type of problem and knew immensely more 
about the classification than I. So was born \Prob2, and the collaboration of Guralnick-Thompson. 

\subsubsection{Progress on \Prob2} John showed me the initial list from his
work  with Bob on the affine group case. The display mode was groups presented by branch
cycle generators: an absolute Nielsen class (\S\ref{nc1}).    I noted three  
degree 25 rational functions requiring just the {\sl branch cycle argument\/} \cite[p.~62]{FrHFGG}
to see they had the Schur cover property. (You measure this by distinguishing between its arithmetic
and geometric monodromy groups.)  meant it did not
come from elliptic curves, or twists of cyclic or Chebychev polynomials.  

We now have a list of Nielsen classes sporadic for the Schur
property (Schur-sporadic; \cite[Thm.~1.4]{GMS}). \cite[\S7.2.1]{FrExcTow} uses this as we did
M\"uller's polynomial 0-sporadic list in this paper. 

{\sl \Prob2 in John's form is true}. There are only finitely many sporadic genus 0 groups. 
Most major contributors are in this chronological list: \cite{GT90}, \cite{As90}, 
\cite{LS91},  \cite{S91}, \cite{GN92}, \cite{GN95}, \cite{LSh99}, and  \cite{FM01}. Reverting to
the primitive case parceled the task through the 5-branch Aschbacher-O'Nan-Scott classification of
primitive groups
\cite{AschOnanScott}.  

\subsubsection{Guralnick's optimistic conjecture} \label{nonSporadics}  Yet, there is an
obvious gap between the early papers and the two at the end. The title of  \cite{LSh99} reveals it
did not list examples precisely as John did at the beginning. We can't yet expect  the
{\sl exceptional\/} Chevalley groups to fall easily to such explicitness; you can't grab your
favorite permutation representation with them. Still, composition factors are one thing, actual 
genus 0 primitive monodromy
groups another. 

\begin{defn} We say
$T:G\to S_n$, a faithful permutation representation,  with properties \eqref{0-sporadica} and
\eqref{0-sporadicb} is {\sl
$0$-sporadic\/}. \end{defn} 
Denote  $S_n$  on 
unordered $k$ sets of $\{1,\dots,n\}$ by $T_{n,k}:S_n\to S_{\scriptscriptstyle{{n\choose k}}}$
by ($T_{n,1}$ the standard action). Alluding to $S_n$ (or $A_n$) with $T_{n,k}$ nearby
refers to this presentation. In \eqref{0-sporadicb}, 
$V_a=(\bZ/p)^a$ ($p$ a prime). Use \S\ref{nonempNielsen2} for semidirect product in the 
$T_{V_a}$ case on points of $V_a$; $C$ can be $S_3$. For the second $(A_n,T_{n,1})$ case, $T:G\to
S_{n^2}$. 
\begin{triv} \label{0-sporadica} $(G,T)$ is the monodromy group of a primitive (\S\ref{nc1})
compact Riemann surface cover
$\phi:X\to\prP^1_z$ with
$X$ of genus $0$.  \end{triv}  
\begin{triv} \label{0-sporadicb} $(G,T)$ is not in this list of group-permutation types. \end{triv}
\!\begin{itemize} \item  $(A_n,T_{n,1})$: $A_n\le G\le S_n,\text{ or } A_n\times A_n \xs \bZ/2 \le G
\le S_n\times S_n\xs \bZ/2$. 
\item $(A_n,T_{n,2})$: $A_n\le G\le S_n$. 
\item $T_{V_a}$: $G= V\xs C$,\  $a\in \{1,2\}$, $|C|=d\in \{1,2,3,4,6\}$ and $a = 2$
only if
$d$ does not divide $p-1$. 
\end{itemize}

Rational functions $f\in \bC(x)$ represent 0-sporadic groups by 
$f:\prP^1_x\to \prP^1_z$. We say 
$(G,T)$ is {\sl polynomial 0-sporadic}, if some $f\in\bC[x]$ represents it.

\begin{defn} Similarly, we say $(G,T)$ is $g$-sporadic if \eqref{0-sporadica} holds replacing genus
0 by genus $g$.\end{defn} For $g$-sporadic, the list of \eqref{0-sporadicb} is too large.
\cite[Thm.~4.1]{GMS} tips off the adjustments for $g=1$ (\S\ref{qualvsquan}). For,  $g > 1$,
$g$-sporadic groups should be just
$A_n\le G\le S_n$ of $T_{n,1}$ type, and cyclic or dihedral groups. 

\cite{GSh04} has 0-sporadics with an $A_n$ component. \cite{FGMa02} has 0-sporadics
groups with a rank 1 Chevalley group component. Magaard has
written an outline of the large final step: Where components are higher rank Chevalley groups. 
Like the classification itself, someone going after a concise list of such examples for a 
particular problem will have difficulty culling the list for their problem. Various
lists of the 0-sporadics appear in many papers. 

If someone outside group theory comes upon a problem suitable for the monodromy method
or some other, can they go to these papers, look at the lists and finish their
projects? 
\cite{Solomon} has anecdotes and
lists on the classification that many non-group theorists can read. Pointedly,
however, is it sufficient to allow you or I  to have replaced any contributor to \cite{GMS}?
Unlikely! How about to read
\cite{GMS}? Maybe! Yet, not without considerable motivation. 

One needs 
familiarity with the relation between primitive subgroups of $S_n$ and simple groups, the
description from Aschbacher-O'Nan-Scott.  That does not rely on the classification.
Rather, it treats simple group appearances as a black box. To decide  if there are simple groups
satisfying extra conditions contributing to the appearance of a particular primitive group,
you must know special information about the groups in \cite [p.~341]{Solomon}. 
  
\subsubsection{Qualitative versus quantitative} \label{qualvsquan} John's desire for documenting
0-sporadic groups added many pages  to the literature. What did particular examples do? How does
one present specific examples to be useful? We have been suggesting   
\S\ref{MullList} as a model. Mueller's list 
reveals just how relevant was Davenport's Problem for nailing polynomial 0-sporadic groups. Other examples, like
\cite{FGS} and \cite{GMS}, use a condition about a group normalizing $G$. This eliminated much of
the primitive group classification. A seeker after applying the same method may find they, too,
have such a useful condition, making it unnecessary to rustle through many of the lists like
\cite{FGMa02} or
\cite{GSh04}. We explain. Warning: An $(S_5,T_{5,2})$ sporadic case 
occurs in answer a simple question about all indecomposable $f\in \bZ[x]$ on Hilbert's
irreducibility theorem. Nor can you just look at a 0-sporadic Nielsen class to decide if it has an 
arithmetic property to be HIT-sporadic (a name coined for this occasion; \cite[\S2-\S3]{DFr99}).
From experience, these sporadics, in service of a real problem, attract all the attention. 

\cite[Thm.~1.4]{GMS}  classified 0-sporadics with  the Schur property (Schur-sporadics) over number
fields.
In \cite[p.~586]{RETExp} we used the Schur property to show how to handle an entwining between  
arithmetic and geometric monodromy groups. Group theory setup: Two subgroups $G\le \hat G\le
S_n$  of $S_n$ have this property.  
\begin{triv} \label{taucoset} There is a
$\tau\in
\hat G\setminus G$ so that $g$ in the coset 
$G\tau$ implies $g$ fixes precisely one integer (see Ex.\ref{afE}).\end{triv}  
\noindent \cite{GMS} calls
our  Schur covering property, arithmetic exceptionality.)

\cite[Thm.~1.4, c)]{GMS} has the list where the genus of the Galois closure exceeds 1.
These {\sl are\/} the Schur-sporadics: Only finitely many 0-sporadic groups occur. Yet, John's
original problem posed sporadic to mean the composition factors included other than
cyclic or alternating groups. All three types of degree 25 alluded to above were not sporadic from
this criterion. Indeed, the only nonsporadic from this criterion in the whole list were the groups 
$\PSL_2(n)$ with
$n=8$ and 9, and these had polynomial forerunners from \cite{primPol}. The more optimistic
conjecture of \S\ref{nonSporadics} emerged because of the care John insisted upon. 
    
Also, how about the nonsporadic appearances? 
Would you think {\sl big theorem\/} when you hear of a study of dihedral groups? 
Yet,
\cite{SeAbell-adic} is a big theorem.  It is the arithmetic of special four branch
point dihedral covers. The kind we call {\sl involution\/}: They are in the Nielsen class
$\ni(D_{p^{k+1}},\bfC_{2^4})$, four repetitions of the involution conjugacy class in $D_{p^{k+1}}$.
\cite[\S5.2-5.3]{FrExcTow} connects this to exceptional covers (think Schur covering property).   
\cite[App.~C]{FrExcTow} shows alternating and dihedral groups are dual 
for arithmetic questions about monodromy group covers.  

\begin{exmp}[Rational functions and the open image theorem] \label{afE} 
Notice that $D_p\le
\bZ/p\xs (\bZ/p)^*\le S_p$ satisfies the criterion of \eqref{taucoset}. The goal is to describe
rational functions
$f\in K(x)$ with
$K$ some number field so the Galois closure group 
$\hat G_f$ (resp.~$G_f$) of $f(x)-z$ over $K(z)$ (resp.~$\bC(z)$) gives such a pair. The only
simplification is that $f$ can't decompose into lower degree polynomials over $K$. When $E$ is an
elliptic curve without complex multiplication in \S\ref{nonempNielsen2} produces  $f$ 
indecomposable over $K$, but decomposable over $\bC$. This is one of the two nonsporadic cases
where the Galois closure cover for $f:\prP^1_w\to \prP^1_z$ has genus 1.
\cite[Thm.~1.4, b)]{GMS} has the complete list where the Galois closure cover has genus 1. A 
reader new to this will see some that look sporadic. Yet, those come from elliptic curves
and topics like complex multiplication. They are cases where $K=\bQ$ and there are special isogenies over $\bQ$ defined by a $p$-division point not over
$\bQ$. \end{exmp} 

\subsubsection{Monstrous Moonshine uniformizers} \label{MMoon} Recall a rough statement from
{\sl Monstrous Moonshine\/}. Most genus 0 quotients from modular subgroups of $\PSL_2(\bZ)$
have  uniformizers from $\theta$-functions that are  automorphic functions on the upper half plane.
This inspired conjecture, to 
which John significantly contributed \cite{Thompson1}, gives away John's intense desire  to see the
genus 0 monodromy covers group theoretically. A recent Fields  Medal to Borcherds on this topic
corroborates the world's interest in  genus 0 function fields, if the uniformizer has {\sl
significance}.  

The Santa Cruz conference of 1979  alluded to on
\cite[p.~341]{Solomon} (\cite{feit} and \cite{RETExp} came from there) had intense
discussions of Monstrous Moonshine. This was soon after a suggestion by A.~Ogg: He noticed that
primes $p$ dividing the order of the Monster (simple group; denote it by $M$)  are those where the
function field of  the  normalizer of $\Gamma_0(p)$ in $\PSL_2(\bR)$ has genus zero. A.~Pizer was
present: He contributed that those    primes  satisfy a certain conjecture of Hecke relating
modular forms  of weight 2 to 
quaternion algebra $\theta$-series \cite{Pizer}. Apparently Klein, Fricke and Hecke had recognized
the  problem of finding the function field generator of genus 0 quotients of 
the upper half-plane, not necessarily given by congruence subgroups of  
$\PSL_2(\bZ)$.  It seems somewhere in the literature 
is the phrase {\sl genus 0 problem\/} attached to a specific Hecke 
formulation.

Thompson concocted a relation between $q=e^{2\pi i \tau}$-expansion coefficients 
of $j(\tau) = 
q^{-1} + \sum_{n=0}^\infty u_n q^n$ ($u_0=744$, $u_1=196884$, $u_2=21493760$, 
$u_3=864299970$, $u_4=20245856256$, $u_5 = 333202640600$) and  irreducible 
characters of $M$. At the time, the Monster hadn't been proved to exist, and even if it did, some
of its character degrees weren't shown for certain. (\cite{Thompson1} showed if the Monster
existed, these properties uniquely defined it.) John noted the coefficients
listed for
$j$ are  sums of positive 
integral  multiples of these. With this  data he conjectured a 
$q$-expansion with coefficients in Monster characters with these properties \cite{Thompson2}: 
\begin{itemize} \item the $q$ expansion of $j$ is its
evaluation at 1; and 
\item the other genus 0 modular related covers have uniformizers from its evaluation
at the other conjugacy classes of $M$. \end{itemize} 

\cite[p.~28]{Ray} discusses those automorphic forms on the upper half plane with product
expansions following Borcherd's characterization. Kac-Moody
algebras give automorphic forms with a product expansion. The construction of a Monster Lie
Algebra that gave $q$-expansions matching that predicted by Thompson is what \wsp to a novice
like myself \wsp looks like the main story of
\cite{Ray}. 

Is there any function, for example on any of the genus 0 spaces from this paper that
vaguely has a chance to be like such functions? Between \cite[App.~B.2]{BFr} and \cite{SeTheta}
one may conclude the following discussion. 

All components of the 
$\sH_{p,k}^\rd\,$s  in \S\ref{projNCs} have $\theta$-nulls canonically attached to their
moduli definition. So do many of the quotients between  $\sH_{p,k+1}^\rd$ and $\sH_{p,k}^\rd$. 
For the
$\sH_{p,k}^\rd\,$s we really do mean
$\theta$-nulls defined by analytically continuing a $\theta$ function on the Galois cover
$\phi_\bp:X_\bp \to
\prP^1_z$ attached  to $\bp\in \sH_{p,k}^\rd$, then evaluating it at the origin. Usually such
$\theta$-nulls have a character attached to them. Here that would be related to the genus of
$X_\bp$. So, we cannot automatically assert these $\theta$-nulls are automorphic on the upper
half plane (compare with genus 1 versions in \cite{FKraTheta}).
\cite{Si63} (though hard to find) presents the story of $\theta$-functions defined by
unimodular quadratic forms. These do define automorphic functions on the upper half plane. 

The discussion for
$\sH_{2,1}$ alluded to two genus 1 and two genus 0 components. The $\theta$  for the genus 0
components is an odd function. So its $\theta$-null will be identically 0. For the genus 1,
components, however, it is even, and both those components cover (by a degree 2 map) a genus 0
curve between
$\sH_{2,1}$ and
$\sH_{2,0}$. That is one space we suggest for significance. 

\subsubsection{Strategies for success} 
\cite{brown} has a portrait of Darwin as a man of
considerable self-confidence, one who used many strategies to further his  evolution theory. Though
this is contrary to other biographies of Darwin, the case is  convincing. Darwin was a voluminous
correspondent, and his 14,000$^+$ letters are recorded  in many places. Those letters reflected his
high place in the scientific  community. They often farmed out to his
correspondent the task of completing a  biological search, or even a productive experiment. So,
great was Darwin's reputation that  his correspondents allowed him to travel little (in later
life), and yet accumulate great  evidence for his mature volumes. Younger colleagues (the famous
Thomas Huxley, for  example) and even his own family (his son Francis, for instance) presented him
a protective  team and work companions. The success of the theory of evolution owes much to the
great  endeavor we call Charles Darwin. 

\cite{Rowe} reflects on the different way mathematical programs achieve success in his 
review of essays that touch on the growth of US mathematics into the international 
framework. He suggests the international framework that \cite{Parshall} touts may not
be the most  compelling approach to analyzing mathematical success. 

\begin{quote} Recently, historians of science have tried to understand how \dots locally gained
knowledge produced by research schools becomes {\sl universal}, a process that involves analyzing
all the various mechanisms that produce consensus and support within broader scientific
networks and communities. Similar studies of mathematical schools,
however, have been lacking, a circumstance \dots partly due to the
prevalent belief that mathematical knowledge is from its very inception
universal and \dots stands in no urgent need to win converts. \end{quote}

There are two genus 0 problems: \Prob1 and \Prob2. They seem very different. Yet, they are two
of the resonant contributions of John Thompson, outside his first area of renown.
His influence on their
solutions and applications is so large, you see I've struggled to complete their context. The
historian remarks intrigue me for it would be valuable to learn, along their lines, more about our
community.  
\providecommand{\bysame}{\leavevmode\hbox to3em{\hrulefill}\thinspace}

\end{document}

%% file: macroltx.tex
\newcommand{\teq}{\arabic{section}.\arabic{equation}}
\newcommand{\teql}{\Alph{section}.\arabic{equation}}

\newcommand{\sqr}[2]{{\vcenter{\vbox{\hrule height.#2pt\hbox{\vrule width.#2pt
height#1pt \kern#1pt\vrule width.#2pt}\hrule height.#2pt}}}}

\newcounter{eqcount}

\newenvironment{edesc}{\refstepcounter{equation}\begin{enumerate}}%
{\end{enumerate}}
\newenvironment{triv}{\refstepcounter{equation}\begin{list}%
{{\hbox{\rm(\teq)\ }}} \item }{\end{list}}

\newcommand{\ring}[1]{{\mathbb #1}}
\newcommand\bZ{{\ring{Z}}}

\newcommand\bC{{\ring{C}}} \newcommand\bR{{\ring{R}}}
\newcommand\bF{{\ring{F}}} \newcommand\bQ{{\ring{Q}}}

\newcommand{\csp}[1]{{\mathbb #1}}
\newcommand{\tsp}[1]{{\mathcal #1}}

\newcommand{\prP}{\csp{P}}

\newcommand{\sO}{{\tsp{O}}} 
\newcommand{\sQ}{\tsp{Q}}
\newcommand{\sP}{{\tsp {P}}} 
 
 \newcommand{\sH}{{\tsp {H}}}
 
 \newcommand{\sD}{{\tsp {D}}}
\newcommand{\sG}{{\tsp {G}}}

\newcommand{\eql}[2]{{\rm (\ref{#1}\ref{#2})}} 

\newcommand{\vect}[1]{{\pmb #1}} 
 \newcommand{\bg}{\vect{g}}
 \newcommand{\be}{{\vect{e}}}
\newcommand{\bp}{{\vect{p}}} 
\newcommand{\bv}{{\vect{v}}} 
 \newcommand{\bz}{{\vect{z}}}
  
\newcommand{\bh}{{\vect{h}}}  
\newcommand{\row}[2]{{#1_1,\ldots,#1_{#2}}}

\newcommand{\smatrix}[4]{{\big(\begin{array}{cc}
\!\lower2pt\hbox{$\scriptstyle#1$} &\lower2pt\hbox{$\scriptstyle#2$}\!
\\\! \raise2pt\hbox{$\scriptstyle#3$} &\raise2pt\hbox{$\scriptstyle#4$}
\!\end{array}\big)}}

\newcommand{\texto}[1]{{\textr{#1}}}
\newcommand{\GL}{\texto{GL}} 
 \newcommand{\ind}{\texto{ind}}
\newcommand{\PSL}{\texto{PSL}} \newcommand{\PGL}{\texto{PGL}}
 
 \renewcommand{\ni}{\texto{Ni}}
 
\newcommand{\textr}[1]{{\text{\rm #1}}}
 \newcommand{\ord}{\textr{ord}}
\newcommand{\abs}{\textr{abs}}  
 
 \newcommand{\inn}{\textr{in}}

\newcommand{\norm}{{\triangleleft\,}}

\newcommand{\rd}{\texto{rd}}

\newcommand{\textb}[1]{{\text{\bf #1}}}
\newcommand{\bfC}{{\textb{C}}}
\newcommand{\longmapright}[2]{\smash{\mathop{\hbox to
#2pt{\rightarrowfill}}\limits^{#1}}}
\newcommand{\longmapleft}[2]{\smash{\mathop{\hbox to
#2pt{\leftarrowfill}}\limits^{#1}}}

\newcommand{\mapright}[1]{\smash{\mathop{\longrightarrow}\limits^{#1}}}

\newcommand{\np}{{+}}

\newcommand{\lrang}[1]{{\langle #1\rangle}}

\newcommand{\eqdef}{\stackrel{\text{\rm def}}{=}}

\newfont{\sevenrm}{cmr7}
\newfont{\bsevenrm}{cmbx7}
\newfont{\mathseven}{cmsy7}
\newfont{\bigmath}{cmsy10 scaled 1200}
\newfont{\fiverm}{cmr5}
\newfont{\bfiverm}{cmbx5}
\newfont{\hel}{cmbx10 scaled 1200}
\newfont{\eu}{eufb10}
\newfont{\sseu}{eufm5}
\newfont{\seu}{eufm7}
\newfont{\Cal}{cmmib10}
\newfont{\sCal}{cmmib7}
\newfont{\zch}{eusb10}

\theoremstyle{plain}
\newtheorem{thm}{Theorem}[section] 

\newtheorem{prop}[thm]{Proposition}

\theoremstyle{definition}
\newtheorem{defn}[thm]{Definition}
\newtheorem{exmp}[thm]{Example}

\newtheorem{prob}[thm]{Problem}

\theoremstyle{remark}
\newtheorem{rem}[thm]{Remark}

\newcommand{\xs}{\times^s\!}
\newcommand{\wsp}{{$\,$---$\,$}} 




%% file: thomp-genus0.bbl
\begin{thebibliography}{ColCa99}

\bibitem[AFH03]{AFH} W.~Aitken, M. Fried and L.~Holt, \emph{Davenport Pairs over finite
fields}, in proof, PJM, Dec. 2003.

\bibitem[As90]{As90} M.~Aschbacher, \emph{On conjectures of Guralnick and Thompson}, 
J.~Algebra \textbf{1990 \#2}, 277--343. 

\bibitem[AS85]{AschOnanScott}
M.~Aschbacher and L.~Scott, \emph{Maximal subgroups of finite groups}, J. Alg.
  \textbf{92} (1985), 44--80.

\bibitem[A57]{eartin} E.~Artin, \emph{Geometric Algebra}, Interscience tracts in
pure and applied math.~\textbf{3}, 1957. 

\bibitem[BFr02]{BFr} P.~Bailey and M.~Fried, \emph{Hurwitz monodromy, spin 
separation and higher levels
of a Modular Tower}, in  Proceedings of Symposia in Pure Mathematics {\bf 70} 
(2002) editors
M.~Fried and Y.~Ihara, 1999 von Neumann
Conference on Arithmetic Fundamental Groups and Noncommutative Algebra, August 
16-27,
1999 MSRI, 79--221. 

\bibitem[B03]{brown} J.~Browne, \emph{Charles Darwin: The Power of Place}, Knopf, 2003.

\bibitem[Ca56]{carmichael} R.D.~Carmichael, \emph{Introduction to the theory groups of
finite order}, Dover Pub.~1956. 

\bibitem[CoCa99]{couveignes2} J.-M.~Couveignes and P.~Cassou-Nogu\`es,
\emph{Factorisations explicites de $g(y)-h(z)$},  
Acta Arith.~\textbf{87} (1999), no. 4, 291--317.

\bibitem[Co00]{couveignes1} J.-M.~Couveignes, \emph{Tools for the computation of families of
covers}, in Aspects of Galois Theory, Ed: H.~V\"olklein, Camb.~Univ.~Press, LMS Lecture Notes
\textbf{256} (1999), 38--65. The author informs me that due to a mistake of the editors, three pages
have been cyclically permuted. The correct order of pages can be found following the numbering of
formulae.

\bibitem[CKS76]{CuKanSe}
C.W.~Curtis, W.M.~Kantor and G.M. Seitz, \emph{The
2-transitive permutation
  representations of the finite {C}hevalley groups}, TAMS \textbf{218} (1976),
  1--59.

\bibitem[DLS61]{DLSf-g} H.~Davenport, D.J.~Lewis and A.~Schinzel, \emph{Equations of the form
$f(x)=g(y)$}, Quart.~J.~Math.~Oxford \textbf{12} (1961), 304--312. 

\bibitem[DS64]{DSf-g} H.~Davenport, and A.~Schinzel, \emph{Two problems concerning polynomials},
Crelle's J.~\textbf{214} (1964), 386--391. 

\bibitem[DDE04]{DDE} P.~Debes and M.~Emsalem, \emph{Harbater-Mumford components and
Towers of Moduli Spaces}, Presentation by M.~Emsalem at Graz, July 2003, preprint Jan. 2004. 

\bibitem[DFr94]{DFr} P.~Debes and M.D.~Fried, \emph{Nonrigid situations
in constructive Galois theory}, Pacific Journal \textbf{163 \#1}
(1994), 81--122.

\bibitem[DFr99]{DFr99} P.~Debes and M.D.~Fried, \emph{Integral
specializations of families of  rational  functions}, PJM \textbf{190} (1999), 45--85. 
  
\bibitem[FaK01]{FKraTheta}
H.~Farkas and I.~Kra, \emph{Theta Constants, Riemann
Surfaces and the Modular Group}, AMS graduate text
series \textbf{37}, 2001.

\bibitem[Fe73]{feit1} W.~Feit, \emph{On symmetric balanced incomplete block designs with doubly
transitive automorphism groups}, J.~of Comb.; Series A \text{bf} (1973), 221--247. 

\bibitem[Fe80]{feit} W.~Feit,  \emph{Some consequences of the classification of finite
simple groups}, Proceedings of Symposia in  Pure Math: Santa Cruz  Conference on Finite
Groups, A.M.S.  Publications
\textbf{37} (1980),  175--181.

\bibitem[Fe92]{feit2} W.~Feit,  \emph{E-mail to Peter M\"uller}, Jan. 28, 1992. 

\bibitem[Fr70]{FrSchur}
M.D.~Fried, \emph{On a conjecture of Schur}, Mich.~Math.~J.
   \textbf{17} (1970), 41--45.
 
\bibitem[Fr73]{FrRedPol}
M.D.~Fried, \emph{The field of definition of function fields
and a problem in
   the reducibility of polynomials in two variables}, Ill.~J.~of Math.
   \textbf{17} (1973), 128--146.

\bibitem[Fr77]{FrHFGG}
M.~Fried, \emph{Fields of definition of function fields and {H}urwitz families
  and groups as {G}alois groups}, Communications in Algebra \textbf{5} (1977),
  17--82.

\bibitem[Fr78]{FrGGCM}
M.~Fried, \emph{{G}alois groups and {C}omplex {M}ultiplication}, 
T.A.M.S. {\bf 235} (1978) 141--162. 

\bibitem[Fr80]{RETExp} M.D.~Fried,  \emph{Exposition
on  an Arithmetic-Group Theoretic Connection via  Riemann's
Existence  Theorem}, Proceedings of Symposia in  Pure Math:
Santa Cruz  Conference on Finite Groups, A.M.S.  Publications
\textbf{37} (1980),  571--601.

\bibitem[Fr95a]{Fr95} M.~Fried, \emph{Extension of Constants, Rigidity, and the 
Chowla-Zassenhaus  Conjecture}, Finite Fields and their applications, 
Carlitz volume 
\textbf{1} (1995), 326--359.

\bibitem[Fri95b]{FrMT}
M.~D. Fried, \emph{Modular towers: {\sl Generalizing the relation between
  dihedral groups and modular curves\/}}, Proceedings AMS-NSF Summer
  Conference, \textbf{186}, 1995, Cont. Math series, Recent Developments in the
  Inverse Galois Problem, 111--171.

\bibitem[Fr99]{Fr-Schconf}
M.D.~Fried, \emph{Separated variables polynomials and moduli spaces}, Number
   Theory in Progress (Berlin-New York) (ed. J.~Urbanowicz K.~Gyory, H.~Iwaniec,
   ed.), Walter de Gruyter, 1999, Proceedings of the Schinzel Festschrift,
   Summer 1997: Available from www.math.uci.edu/$\tilde{\phantom u}$mfried/\#math, 169--228.

\bibitem[Fr02]{MSRIvol} M.D.~Fried, \emph{Prelude: Arithmetic fundamental groups and
noncommutative algebra}, Proceedings of Symposia in Pure Mathematics, \textbf{70} (2002)
editors M.~Fried and Y.~Ihara, 1999 von Neumann Conference on Arithmetic Fundamental Groups
and Noncommutative Algebra, August 16-27, 1999 MSRI, vii--xxx.

\bibitem[Fr04]{FrExcTow} M.D.~Fried, \emph{Extension of constants series and towers of
exceptional covers}, preprint available in list at 
 www.math.uci.edu/$\tilde{\phantom u}$mfried/\#math or  www.math.uci.edu/$\tilde{\phantom
 u}$mfried/psfiles/exctow.html. 

\bibitem[Fr05]{FB} M.D.~Fried, \emph{Riemann's existence theorem: An elementary approach to
moduli}, Chaps. 1--4 available at www.math.uci.edu/$\tilde{\phantom  u}$mfried/\#ret. 

\bibitem[FGS93]{FGS}
M.D.~Fried, R.~Guralnick and J.~Saxl, \emph{Schur covers and Carlitz's
   conjecture}, Israel J. \textbf{82} (1993), 157--225.

\bibitem[FrS04]{schurtype} M.D.~Fried and D.~Semmen, \emph{Schur multiplier types and Shimura-like
systems of  varieties}, preprint available in the list at www.math.uci.edu/$\tilde{\phantom 
u}$mfried/\#mt or  www.math.uci.edu/$\tilde{\phantom  u}$mfried/psfiles/schurtype.html. 

\bibitem[FV91]{FrVMS}
M.~Fried and H.~V{\"o}lklein, \emph{The inverse {G}alois
problem and rational
  points on moduli spaces} , Math.~Annalen
\textbf{290} (1991), 771--800.

\bibitem[FMa01]{FM01} D.~Frohardt and K.~Magaard, \emph{Composition Factors of Monodromy Groups}, 
Annals of Math.~\textbf{154} (2001),  1--19

\bibitem[FGMa02]{FGMa02} D.~Frohardt, R.M.~Guralnick and K.~Magaard, \emph{Genus 0 actions of
groups of Lie rank 1}, in  Proceedings of Symposia in Pure Mathematics {\bf 70} 
(2002) editors
M.~Fried and Y.~Ihara, 1999 von Neumann
Conference on Arithmetic Fundamental Groups and Noncommutative Algebra, August 
16-27,
1999 MSRI, 449--483.  

\bibitem[GHP88]{GHP} L.~Gerritzen, F.~Herrlich, and M.~van der Put, \emph{Stable $n$-pointed trees
of projective lines}, Ind.~Math.~\textbf{50} (1988), 131--163. 

\bibitem[GN92]{GN92} R.M.~Guralnick, \emph{The genus of a permutation group}, in Groups,
Combinatorics and Geometry, Ed: M.~Liebeck and J.~Saxl, LMS Lecture Note Series \textbf{165}, CUP,
Longdon, 1992. 
 
\bibitem[GMS03]{GMS} R.~Guralnick, P.~M\"uller and J.~Saxl,
\emph{The rational function analoque of a question of Schur and
exceptionality of permutations representations}, Memoirs of the AMS
\textbf{162} 773 (2003), ISBN 0065-9266. 

\bibitem[GN95]{GN95} R.M.~Guralnick and M.G.~Neubauer, \emph{Monodromy groups of branched coverings:
the generic case}, Proceedings AMS-NSF Summer
  Conference, \textbf{186}, 1995, Cont. Math series, Ed: M.~Fried Recent Developments in the
  Inverse Galois Problem, 325--352.

\bibitem[GSh04]{GSh04} R.M.~Guralnick and J. Shareshian, \emph{Symmetric and Alternating Groups as 
Monodromy Groups of Riemann Surfaces I}, preprint.

\bibitem[GT90]{GT90} R.M.~Guralnick and J.G.~Thompson, \emph{Finite groups of genus 0}, J.~Algebra
\textbf{131} (1990), 303--341. 

\bibitem[LS91]{LS91} M.~Liebeck and J.~Saxl, \emph{Minimal degrees of primitive permutation groups, with an application to monodromy groups of covers of Riemann surfaces}, PLMS \textbf{(3) 63} (1991), 266--314. 

\bibitem[LSh99]{LSh99} M.~Liebeck and A.~Shalev, \emph{Simple groups, permutation groups, and probability}, 497--520. 

\bibitem[MM99]{MM}  
G.~Malle and B.H.~Matzat,
\emph{Inverse {G}alois Theory},
ISBN 3-540-62890-8, Monographs in Mathematics, 
Springer,1999. 

\bibitem[M{\"u}98a]{MeuDavI}
P.~M{\"u}ller, \emph{Kronecker conjugacy of polynomials}, TAMS \textbf{350}
   (1998), 1823--1850.

\bibitem[Mu95]{primPol} 
P.~M\"uller, \emph{Primitive monodromy groups of polynomials},  
Recent developments in the inverse Galois problem  AMS, Cont. Math. Series  Editor: Michael
(1995), 385--401.  

\bibitem[PaR02]{Parshall} K.~Parshall and A.~Rice, \emph{Mathematics unbound: The evolution of an
international mathematical research community, 1800--1945}, History of Math.~vol. \textbf{23},
AMS/LMS, Prov.~RI, 2002. 

\bibitem[P78]{Pizer} A.~Pizer, \emph{A Note on a Conjecture of Hecke}, PJM \textbf{79}
(1978), 541--548.

\bibitem[Ro03]{Rowe} D.~Rowe, \emph{Review of \cite{Parshall}}, BAMS \textbf{40 \#4} (2003),
535--542. 

\bibitem[Se68]{SeAbell-adic}
J.-P. Serre, \emph{Abelian $\ell$-adic representations and elliptic curves},
  1st ed., McGill University Lecture Notes, Benjamin, New York $\bullet$
  Amsterdam, 1968, in collaboration with Willem Kuyk and John Labute.

\bibitem[Ser90b]{SeTheta}
J.-P. Serre, \emph{Rev\^etements a ramification impaire et
  th\^eta-caract\'eristiques}, C.~R.~Acad. Sci. Paris \textbf{311} (1990),
  547--552.

\bibitem[S91]{S91} T.~Shih, \emph{A note on groups of genus zero}, Comm.~Alg. \textbf{19} (1991), 2813--2826. 

\bibitem[Si63]{Si63} C.L.~Siegel, \emph{Analytic Zahlentheorie II Vorlesungen}, gehalten im
Wintersemester 1963/64 an der Universit\"at G\"ottingen, mimeographed notes. 

\bibitem[To79a]{Thompson1} J.G.~Thompson, \emph{Finite groups and modular functions},  BLMS
\text{11 (3)}  (1979), 347--351. 
\bibitem[To79b]{Thompson2} J.G.~Thompson, \emph{Some Numerology between the Fischer-Griess Monster and the elliptic modular 
function},   BLMS
\text{11 (3)}  (1979), 340--346. 

\bibitem[Ra00]{Ray} U.~Ray, \emph{Generalized Kac-Moody algebras and some related topics}, BAMS
\textbf{38 \#1}, 1--42. 


\bibitem[So01]{Solomon} R.~Solomon, \emph{A brief history of the classification of finite simple
groups}, BAMS \textbf{38 \#3} (2001), 315--352. 

\bibitem[V{\"o}96]{VB}
H.~V{\"o}lklein, \emph{Groups as {G}alois 
{G}roups} {\bf 53},
Cambridge Studies in Advanced Mathematics,
Camb.~U.~Press, Camb.~England, 1996.


\bibitem[We99]{Wewers} S. Wewers, \emph{Deformation of tame admissible covers of curves}, 
in Aspects of Galois Theory,
Ed: H.~V\"olklein, Camb.~Univ.~Press, LMS Lecture Notes
\textbf{256} (1999), 239--282.

\bibitem[Wo64]{Wohlfahrt} K.~Wohlfahrt, \emph{An extension of F.~Klein's level concept}, Ill.~J. Math.
\textbf{8} (1964), 529--535. 

\end{thebibliography}
